\documentclass[11pt,reqno]{amsart}
\usepackage{amsmath}
\usepackage{amssymb, latexsym, amsfonts, amscd, amsthm, mathrsfs, enumerate, esint}
\usepackage[usenames,dvipsnames]{color}
\usepackage[all]{xy}

\usepackage{graphicx}
\usepackage{epsfig}
\usepackage{hyperref}

\numberwithin{equation}{section}

\definecolor{purple}{rgb}{0.9,0,0.8}

\definecolor{gray}{rgb}{0.7,0.7,0.7}

\newtheorem{thm}{Theorem}[section]
\newtheorem{cor}[thm]{Corollary}
\newtheorem{lem}[thm]{Lemma}
\newtheorem{ppn}[thm]{Proposition}
\newtheorem{ass}[thm]{Assumption}

\theoremstyle{definition}
\newtheorem{defn}[thm]{Definition}

\newtheorem{remark}[thm]{Remark}
\newcommand{\beq}{\begin{equation}}
\newcommand{\eeq}{\end{equation}}

\newcommand{\E}{\mathbb{E}}

\newcommand{\N}{\mathbb{N}}

	\renewcommand{\P}{\mathbb{P}}

\newcommand{\R}{\mathbb{R}}

\newcommand{\Z}{\mathbb{Z}}

\newcommand{\cC}{\mathcal{C}}

\newcommand{\cE}{\mathcal{E}}
\newcommand{\cF}{\mathcal{F}}

\newcommand{\cJ}{\mathcal{J}}

\newcommand{\cR}{\mathcal{R}}

\newcommand{\bDel}{\mathbf{\Delta}}
\newcommand{\by}{\mathbf{y}}
\newcommand{\bz}{\mathbf{z}}

\newcommand{\wt}{\widetilde}
\newcommand{\wh}{\widehat}

\newcommand{\grad}{\nabla}

\renewcommand{\setminus}{\backslash}

\newcommand{\abbr}[1]{{\sc\lowercase{#1}}}
\begin{document}

\title[Persistence with non-summable correlations]{Persistence of Gaussian processes: \\ 
non-summable correlations}
\date{\today}

\author[A.\ Dembo]{Amir Dembo$^*$}
\author[S.\ Mukherjee]{Sumit Mukherjee}
\address{$^\ast$Department of Mathematics, Stanford University, Building 380, 
Sloan Hall, Stanford, CA 94305, USA}
\address{Department of Statistics, Columbia University,
1011 SSW, 1255 Amsterdam Avenue, New York, NY 10027, USA}

\thanks{$^\ast$Research partially supported by NSF 
grant DMS-1106627.}


\maketitle
\begin{abstract} 
Suppose the auto-correlations of real-valued, centered 
Gaussian process $Z(\cdot)$ are non-negative and 
decay as $\rho(|s-t|)$ for some $\rho(\cdot)$ 
regularly varying at infinity of order 
$-\alpha \in [-1,0)$.
With $I_\rho(t)=\int_0^t \rho(s)ds$ its primitive, 
we show that
the persistence probabilities decay rate of
$ -\log\P(\sup_{t \in [0,T]}\{Z(t)\}<0)$ is precisely of order 
$(T/I_\rho(T)) \log I_\rho(T)$, 
thereby closing the gap between the lower and 
upper bounds of \cite{NR}, which stood as 
such for over fifty years. We demonstrate its 
usefulness by sharpening recent results of 
\cite{Sak} about the dependence on $d$ of 
such persistence decay for the Langevin 
dynamics of certain $\grad \phi$-interface 
models on 
$\Z^d$.
\end{abstract}

\section{Introduction}

Persistence probabilities, namely the asymptotic 
of $\P(\sup_{t\in[0,T]}\{Z(t)\}<0)$ as $T \to \infty$, 
have a fairly long history in probability theory 
with the case of stationary, centered, 
Gaussian processes $Z(t)$ receiving much attention
(c.f. \cite{Slep,NR,Shur,Wat,Pic,SM,DPSZ,Mol,AS} 
and the references therein). 
In particular, for non-negative, stationary auto-correlation $A(s,t)=\E [Z(s) Z(t)]$, it directly follows by an application of Slepian's lemma and sub-additivity 
that the limit
\begin{equation}\label{eq:exp-coef}
b(A):=-\lim_{T\rightarrow\infty}\frac{1}{T}\log \P(\sup_{t\in[0,T]}\{Z(t)\}<0)
\end{equation}
exists in $[0,\infty]$, and it is easy to see that 
$b(A)$ is finite whenever $Z(\cdot)$ has continuous 
sample paths (c.f. Lemma \ref{clm1}). Hereafter for any stationary non-negative correlation function $A(\cdot,\cdot)$ we use the notation $b(A)$ to denote the limit  defined in \eqref{eq:exp-coef}. 

For such processes 
the positivity of $b(A)$, namely the exponential decay of the 
corresponding persistence probabilities, is equivalent to integrability of 
$A(0,\cdot)$, under certain regularity condition on $\tau \mapsto A(0,\tau)$. 

For many processes of interest $\tau \mapsto A(s,s+\tau)$ is non-integrable, with 
$b(A)=0$
(see for example Remark \ref{rem:al-corr},
Corollary \ref{sak_improved} and Remark \ref{rem:fbm}). 
In such cases \eqref{eq:exp-coef} is of limited value, 
and the finer, sub-exponential persistence probability decay 
rate, is of much interest. Indeed, focusing on 
the special case where $A(s,t)$ decays as $|t-s|^{-\alpha}$ 
for some $\alpha\in (0,1]$, already in 1962, 
Newell and Rosenblatt \cite{NR} showed that 
\begin{align}\label{eq:under-one}
T^\alpha \lesssim -\log \P(\sup_{t \in [0,T]}\{Z(t)\}<0)\lesssim& T^{\alpha}\log T,\quad \text{ for } \;\; 0<\alpha<1,\\
\frac{T}{\log T} \lesssim -\log \P(\sup_{t \in [0,T]}\{Z(t)\}<0)\lesssim& T,\qquad \qquad \text{ for } \;\; \alpha=1.
\label{eq:equal-one}
\end{align}
Hereafter, for any non-negative functions $a_1(T)$, $a_2(T)$,
we denote by $a_1(T)\lesssim a_2(T)$ the existence of $C<\infty$, 
possibly depending on the law of $\{Z(\cdot)\}$, such that 
$a_1(T)\le C a_2(T)$ for all $T$ large enough, 
with 
$a_1(T) = \Theta(a_2(T))$ when both 
$a_1(T) \lesssim a_2(T)$ and $a_2(T)\lesssim a_1(T)$.

To the best of our knowledge, the gap between 
the upper and lower bounds of \cite{NR}, as 
in \eqref{eq:under-one}--\eqref{eq:equal-one},
has never been improved. Our main result closes this gap, by determining the correct 
decay rate of the relevant persistence probabilities, in case of asymptotically stationary 
non-negative $A(s,t)$ that are regularly decaying in $|t-s|$ large. To this end, we
shall make use of the following definition.
\begin{defn}
For $\alpha\ge 0$, let $\mathcal{R}_{\alpha}$ denote 
the collection of measurable, regularly varying 
of order ${-\alpha}$ functions 
$\rho:[0,\infty) \mapsto (0,1]$,  i.e. for every $\lambda>0$ one has
$$\lim_{t\rightarrow\infty}\frac{\rho(\lambda t)}{\rho(t)}=\lambda^{-\alpha}.$$
Associate
to each $\rho\in \mathcal{R}_\alpha$ the 
primitive function $I_\rho:(0,\infty)\mapsto (0,\infty)$
such that 
\begin{align}\label{eq:Irho-def}
I_\rho(t):=\int\limits_0^t\rho(s)ds \,,
\end{align}
and the asymptotic \emph{persistence decay rate}
\begin{align}\label{eq:arho-def}
a_\rho(t):= \frac{t \log I_\rho(t)}{I_\rho(t)} \,.
\end{align}
\end{defn}
\begin{thm}\label{gen2}
Suppose the 
centered Gaussian process
$\{Z(t)\}_{t\ge 0}$ has non-negative auto-correlation
\begin{equation}\label{eq:A-def} 
A(s,t) := \frac{\E[Z_s Z_t]}{\sqrt{\E[Z_s^2]\E[Z_t^2]}} \,.
\end{equation}
\begin{enumerate}[(a)]
\item If some $\alpha\in (0,1]$
and $\rho\in \mathcal{R}_\alpha$ with
$I_\rho(\infty)=\infty$, are such that there exists $\widetilde{\eta}>0$ satisfying
\begin{equation}
\label{nonsum2}
\limsup\limits_{t,\tau\rightarrow\infty,\tau \le \widetilde{\eta} t}\frac{ A(t,t+\tau)}{\rho(\tau)}<\infty \,,
\end{equation}
then, 
\begin{equation}\label{eq:ubd-pers}
- \limsup\limits_{T\rightarrow\infty}\frac{1}{a_\rho(T)} \log \P(\sup_{t \in [0,T]}\{Z(t)\}<0) > 0\,.
\end{equation}
\noindent\item 
Suppose further that 
\begin{align}\label{need1}
\lim_{u\downarrow0}\sup_{s\ge 0}\E \big[ \sup_{t\in[s,s+u]}\{Z(t)\} \big] <\infty, 
\end{align}
and there exists $\eta>0$ such that
\begin{align}
\label{nonsum1}\liminf\limits_{t,\tau\rightarrow\infty, \tau\le \eta t}\frac{A(t,t+\tau)}{\rho(\tau)} >0.
\end{align}
Then, 
\begin{equation}\label{eq:theta-pers}
- \log \P (\sup_{t \in [0,T]}\{Z(t)\}<0) = \Theta(a_\rho(T))\,.
\end{equation}
\end{enumerate}
\end{thm}
Theorem \ref{gen2} is proved in Section \ref{proof-gen2}, 
where for the upper bound of part (a) it suffices to consider 
the persistence probabilities over $[rT,T]$ for suitably chosen 
$r\in (0,1)$. We can further split $[rT,T]$ into 
sub-intervals while leaving large enough gaps to ensure that the 
dependence between the restrictions of $Z(t)$ to the different 
sub-intervals, is weak enough for deducing an exponential decay 
of the overall persistence probability in terms of the number 
of such sub-intervals. The more delicate proof of the 
complementary lower bound of part (b) consists of four steps. 
We first rely on Slepian's lemma and the non-negativity of 
the correlation $A(\cdot,\cdot)$ of \eqref{eq:A-def} to 
show that if such lower bound holds for intervals 
$[rT',T']$ with $r \in (0,1)$ fixed and $T'$ 
large enough, then it must also extend to the interval $[0,T]$. 
To verify such a bound for $[rT,T]$, in the second step 
we split it to many sub-intervals, now employing a 
conditioning argument to control the height of the 
end-points of these sub-intervals, provided that the 
conditioned process has non-negative correlations. 
The third step establishes the latter crucial fact, 
thanks to certain properties of any such correlation 
function, the derivations of which are deferred 
to the last step of the proof.
\begin{remark}
For $\alpha \in (0,1)$, upon 
comparing \eqref{eq:arho-def} and \eqref{eq:slow_estimate<1} we see 
that $a_\rho(T) = \Theta(\rho(T)^{-1} \log T)$. So, in this case
our conclusion \eqref{eq:theta-pers} is that the 
persistence probability lower bound of \cite{NR}, 
namely the \abbr{rhs} of \eqref{eq:under-one}, is tight. 
In contrast, $I_\rho(T)$ is not $\Theta (T \rho(T))$ when 
$\alpha=1$, and in particular $\rho(t)=1/(1+t)$ yields
$a_\rho(T) = T (\log \log T)/(\log T)$, with neither the 
upper nor the lower persistence probability bound 
of \cite{NR} then tight.
\end{remark}
\begin{remark}\label{rem:al-corr}
The conclusion \eqref{eq:theta-pers} holds for any 
\emph{stationary} process $\{Z(\cdot)\}$ 
having non-negative auto-correlation $A(0,\tau)=\Theta(\rho(\tau))$ 
for some $\rho \in \cR_\alpha$, $\alpha \in (0,1]$, such that 
$I_\rho(\infty)=\infty$ and {$t \mapsto Z(t)$ has a.s. continuous sample path (which holds for example when
$|\log u|^\eta (1-A(0,u)) \to 0$ as $u \to 0$, for some $\eta>1$,
see \cite[(1.4.3)]{AT}).} Such stationary Gaussian processes of 
algebraically decaying, non-summable correlations  appear 
frequently in the physics literature (see for example  
\cite{GPS,MB1,MB2}, and the excellent survey in \cite{BMS}). 
An interesting open 
problem is to find in this context sufficient 
conditions for the existence of the limit  
\begin{equation}\label{eq:lim-pers}
b_\ast (A) := 
- \lim_{T \to \infty} \frac{1}{a_\rho(T)}  
\log \P (\sup_{t \in [0,T]}\{Z(t)\}<0) \,,
\end{equation}
possibly after replacing $a_\rho(T)$ of \eqref{eq:arho-def} 
by an equivalent function.
\end{remark}

For slowly varying, eventually decaying to zero, correlations (namely, as
in Theorem \ref{gen2}, but with $\rho \in \cR_0$), we next determine 
the rate of decay of persistence probabilities, up to a log factor.
\begin{ppn}\label{decay_slowly_varying}
Suppose in the setting of Theorem \ref{gen2} that 
conditions \eqref{nonsum2}, \eqref{need1} and \eqref{nonsum1} 
hold for some $\rho\in \mathcal{R}_0$ which is eventually 
non-increasing and
$$
\lim_{x\rightarrow\infty}\rho(x)=0 \,.
$$
Then, we have that
\begin{align}\label{eq:decay_slowly_varying} 
a_\rho(T) \lesssim  -\log \P (\sup_{t \in [0,T]}\{Z(t)\}<0) \lesssim a_\rho(T) \log T \,.
\end{align}
\end{ppn}

\begin{remark} Recall that the spectral measure $\mu_A$ of a centered, 
stationary Gaussian process $\{Z(\cdot)\}$ is the unique non-negative measure 
such that 
\begin{align*}
A(0,\tau)=\int_\R e^{-i\lambda \tau}d\mu_A (\lambda)
\,\qquad 
\forall \tau \in \R 
\end{align*}
and in particular, the absolute integrability of $A(0,\cdot)$ implies 
the existence of uniformly bounded density of $\mu_A$. 
Following \cite{DV} treatment of discrete time, centered, stationary Gaussian 
sequences, \cite[Theorem 2.1]{BD} derives the 
Large Deviations Principle (\abbr{ldp}) at speed $T$ and
$C_b(\R)$-topology, for 
$L_T := T^{-1} \int_0^T \delta_{Z(t)} \, dt$,
provided $\mu_A$ has a vanishing at infinity, continuous density. 
This does not imply the \abbr{LDP} for $L_T(-\infty,0)$, so
when going beyond non-negative correlations, the limit in \eqref{eq:exp-coef} might
not exist.
Nevertheless, \cite{FF} provide in this setting sufficient conditions for truly 
exponential decay of the persistence probabilities. Specifically, \cite{FF} shows that
\begin{equation}\label{eq:exp-decay}
-\log \P(\sup_{t\in[0,T]}\{Z(t)\}<0)=\Theta(T)\,,
\end{equation}
if near the origin the corresponding spectral measure $\mu_A$ 
has a bounded away from zero and infinity density, and for some $\delta>0$ 
the integral $\int_\R|\lambda|^\delta \mu_A(d\lambda)$ is finite.
Our proof of Lemma \ref{clm3} shows that any centered, stationary, 
separable Gaussian process $Z(\cdot)$ with absolutely integrable $A(0,\cdot)$, 
has at least exponentially decaying persistence probabilities
(so neither bounded away from zero density 
near the origin nor having
$\int_{\R} |\lambda|^\delta \mu_A (d\lambda) < \infty$,
are required for such exponential decay). 
It further raises the natural question what is 
the precise necessary and sufficient condition 
for having at least exponential decay of persistence 
probabilities of such processes.
\end{remark}

%
%
%
%
%

As an application of Theorem \ref{gen2}, we sharpen 
some of the results of \cite{Sak} about asymptotic persistence probabilities for a certain family of
$\grad \phi$-interface models. Specifically, 
consider the $\R_+ \times \Z^d$-indexed 
centered Gaussian process $\{\phi_t({\bf x})\}$
given 
by the unique strong solution of the corresponding  
(Langevin) system of interacting diffusion processes:
\begin{align}\label{A1}
d\phi_t({\bf x})=\{-\phi_t({\bf x})+\sum_{{\bf y}
\neq {\bf x}}q({\bf y}-{\bf x})\phi_t({\bf y})\}dt+\sqrt{2}dB_t({\bf x}),\quad \phi_0({\bf x})=0.
\end{align}
Here $\{B_t({\bf x})\}_{{\bf x}\in \Z^d}$ is a collection of independent standard Brownian motions, and we make 
the following assumptions about $q:\Z^d\mapsto \R_+$.
\begin{ass}\label{ass-jump-rates} 
The function $q:\Z^d\mapsto \R_+$ 
satisfies the following four conditions:
\begin{enumerate}[(a)]
\item
$q({\bf x})=q(-{\bf x})$,

\item
There exists $R<\infty$ such that $q({\bf x})=0$ whenever $||{\bf x}||_2\ge R$,

\item
$\sum_{{\bf x}\neq {\bf 0}}q({\bf x})=1$,

\item
The additive group generated by 
$\{{\bf x}\in \Z^d:q({\bf x})>0\}$ is 
$\Z^d$.
\end{enumerate}
\end{ass}
Such $\grad\phi$ and other, closely related, models 
received much interest in mathematical physics and probability literature (c.f. 
\cite{Deu,FS,Gar,GOS,Hamm} and the references therein). 
It is not hard to verify that a standard 
approximation argument proves the existence 
of a unique strong solution of (\ref{A1}) 
(that is, a stochastic process 
$\phi_t({\bf x}) \in C([0,\infty),\cE')$ 
for $\cE'=\{{\bf x}: \sum_{i} (1+\|i\|)^{-2p} |x(i)|^2 
< \infty,$ for some $p \ge 1 \}$, adapted to the 
filtration 
$\sigma(B_s({\bf x}): {\bf x} \in \Z^d, s \le t)$
and satisfying \eqref{A1}).
Further, there exists a random walk representation for
the space-time correlations of \eqref{A1} 
(c.f. \cite{DD,Deu2}; see also the references 
therein for other interacting diffusion processes 
admitting a random walk representation 
for their correlations). From this random walk representation we have that the covariance of 
the centered Gaussian process 
$g_t:=\phi_t({\bf 0})$ is 
\begin{equation}\label{eq:gamma-def}
\Gamma^{(q)}(s,t):=\int_{|s-t|}^{s+t}\P(S^{(q)}_u={\bf 0})du\,.
\end{equation}
Here $\{S^{(q)}_u\}_{u\ge 0}$ denotes the continuous time 
random walk on $\Z^d$, starting at $S^{(q)}_0={\bf 0}$
which upon its arrival to any site ${\bf x} \in \Z^d$
waits for an independent, Exponential$(1)$ time, 
then moves with probability $q({\bf y - x})$ 
to ${\bf y} \in \Z^d \setminus \{{\bf x}\}$.
%
%
%
%
%
%
%
%
The correlation of the process $\{g_t\}$ is consequently of the form 
\begin{align}\label{A2}
C_\rho(s,t) =&\frac{I_\rho(s+t)-I_\rho(|s-t|)}{\sqrt{I_\rho(2s)I_\rho(2t)}}\,,
\end{align}
for $I_\rho(\cdot)$ of \eqref{eq:Irho-def}, 
where $\rho(u) := \P(S^{(q)}_u={\bf 0})$ is 
bounded, strictly positive and regularly varying (see the proof of Corollary \ref{sak_improved}).
More generally, replacing $\P(S^{(q)}_u={\bf 0})$ by some other
regularly varying function $\rho$, 
our next theorem provides asymptotic decay of persistence probabilities 
for any centered Gaussian process 
$\{Y_\rho(t)\}_{t> 0}$ having correlation 
$C(s,t):= \E[Y_\rho(t) Y_\rho(s)]$ of the form \eqref{A2} 
for some $\rho\in \cR_\gamma$.
To this end, for $\gamma > 1$ we may utilize 
the corresponding limiting correlation function 
\begin{equation}\label{eq:Arho-def}
\overline{C}_\rho(s,t) := \lim_{k \to \infty} C_\rho(s+k,t+k) 
= 1-\frac{I_\rho(|s-t|)}{I_\rho(\infty)} \,.
\end{equation}
For $\gamma \in [0,1)$ we shall instead consider  
the universal limiting correlation functions 
associated with the 
Lamperti transformation $t=e^u$ (see \cite{Lam}). 
That is, 
\begin{equation}\label{eq:Agam-def}
C^\star_\gamma(v,u):= \lim_{k \to \infty} C_\rho(e^{v+k},e^{u+k})
=\cosh(|u-v|/2)^{1-\gamma}-\sinh(|u-v|/2)^{1-\gamma}.
\end{equation}
The latter functions 
appear  
in the physics literature 
when studying persistence of Gaussian 
processes driven by linear stochastic
differential equations (see \cite{KKMCBS,MB2}).

\begin{thm}\label{all}
Suppose the process $\{Y_\rho(\cdot)\}$ has 
correlation function of the form \eqref{A2}
for some $\rho \in \cR_\gamma$ and let
$I_{\wt{\rho}}(\cdot)$ denote the primitive of 
$\wt{\rho}(s):=s \rho(s)$.
\begin{enumerate}[(a)]
\item
If $\gamma > 2$ or $\gamma=2$ and  
$I_{\wt{\rho}}(\infty)<\infty$, 
then 
\begin{align}\label{sum}
- \lim\limits_{T\rightarrow\infty}\frac{1}{T}\log \P(\sup_{t \in [1,T]}\{Y_\rho(t)\}<0)=
b(\overline{C}_\rho)
\in (0,\infty) \,,
\end{align}
provided $\rho(\cdot)$ is 
uniformly bounded away from zero on compacts. 

\item
If $\gamma\in [0,1)$, then 
\begin{align}\label{sum2}
- \lim\limits_{T\rightarrow\infty}\frac{1}{\log T}\log \P(\sup_{t \in [1,T]}\{Y_\rho(t)\}<0)=b(C^\star_\gamma) \in (0,\infty) \,.
\end{align}
\item
If $\gamma\in (1,2)$ or $\gamma=2$ and $I_{\wt{\rho}}(\infty)=\infty$, then 
\begin{align}
- \log\P(\sup_{t \in [1,T]}\{Y_\rho(t)\}<0) = \Theta(a_{\wt{\rho}}(T)) 
\label{nonsum}
\end{align}
\end{enumerate}
\end{thm}
\begin{remark}
Note that for $\gamma \in [0,1)$ we 
get the same persistence power exponent $b(C^\star_\gamma)$ 
for all $\rho \in \cR_\gamma$ (which  
is not the case when $\gamma >2$).
%
%
\end{remark}

We have the following immediate application of 
Theorem \ref{all} for
$\rho^{(q)}(u):=\P(S^{(q)}_u={\bf 0})$.
\begin{cor}\label{sak_improved}
Fixing $d \in \N$ and $q:\Z^d \mapsto \R_+$ satisfying 
Assumption \ref{ass-jump-rates}, let $g_t=\phi_t({\bf 0})$ 
for $\phi_t({\bf x})$ 
which is the unique strong solution of \eqref{A1}.
\begin{enumerate}[(a)]
\item
If $d=1$ then 
\begin{align}\label{d_eql_1}
-\frac{1}{\log T}\log \P(\sup_{t\in [1,T]}\{g_t\}<0)=b(C^\star_{1/2})\in (0,\infty).
\end{align}

\item
If $d=3$ then 
\begin{align}\label{d_eql_3}
-\log\P(\sup_{t\in [1,T]}\{g_t\}<0)=\Theta(\sqrt{T}\log T).
\end{align}

\item
If $d=4$ then
\begin{align}\label{d_eql_4}
-\log \P(\sup_{t \in [1,T]} \{g_t\} <0)=\Theta\Big(\frac{T\log \log T}{\log T}\Big).
\end{align}

\item
If $d\ge 5$ then 
$\rho^{(q)}(u)=\P(S^{(q)}_u={\bf 0})
\in \mathcal{R}_{d/2}$ and
\begin{align}\label{d_large}
-\lim_{T\rightarrow\infty}\frac{1}{T}\log\P(\sup_{t\in [1,T]}
\{g_t\}<0)=b(\overline{C}_{\rho^{(q)}})\in(0,\infty) \,.
\end{align}
 
\item For jump rates $q_d:\Z^d \mapsto \R_+$ satisfying Assumption \ref{ass-jump-rates}
and any $k \ge 0$, let $G^{(q_d)}_k$ denote the expected occupation time of ${\bf 0}$
during $\{k,k+1,\ldots\}$ by a discrete time random walk of transition probabilities $q_d({\bf y}-{\bf x})$ that starts at $\{\bf 0\}$. Suppose the 
Green functions $G^{(q_d)}_0 \to 1$ as $d \to \infty$ and 
$k^2 G^{(q_d)}_k$ is uniformly bounded over $k \ge 1$ and $d \ge d_0$. Then, 
\begin{align}\label{truly_infty_dim}
\lim_{d\rightarrow\infty} b\big(\overline{C}_{\rho^{(q_d)}}\big) = 1\,.
\end{align}
\end{enumerate}
\end{cor}

\begin{proof}
In view of Assumption \ref{ass-jump-rates},
the convergence $u^{d/2} \P(S^{(q)}_u={\bf 0}) \to c_q$ 
for some finite constant $c_q>0$ readily follows 
from the local \abbr{clt} for $\{S^{(q)}_u\}$, 
as in \cite[Theorem 2.1.3]{LL}. Hence,
comparing \eqref{eq:gamma-def} with \eqref{A2}, 
parts (a)-(d) of the corollary are an immediate 
application of Theorem \ref{all} for
$\rho^{(q)}(u) \in \mathcal{R}_{d/2}$
(indeed, part (a) of Theorem \ref{all} 
takes care of $d \ge 5$, 
part (b) handles $d=1$, while part (c) deals 
with both $d=3$ for which $I_{\wt{\rho}}(T) = \Theta(\sqrt{T})$ 
and $d=4$ for which $I_{\wt{\rho}}(T)=\Theta(\log T)$).
Part (e) is an application \cite[Theorem 1.6]{DM}, the details of 
which are provided in Section \ref{sec-3}.
\end{proof}

\begin{remark}
Corollary \ref{sak_improved} gives the exact order 
of decay for any $d\neq 2$, as well as existence of 
a limiting persistence exponent for $d=1$ and 
$d\ge 5$. In doing so it improves upon the earlier 
results of \cite{Sak}
(where the decay rate is determined for 
$d=1,d\ge 5$ without the existence of a limit, 
and decay rate upper and lower bounds within 
a $\log T$ factor are given for $d=2,3,4$). 
Recall that $\P(S^{(q)}_u=0) = \Theta(1/u)$ when $d=2$, 
hence the process $\{g_t/\sqrt{\E g_t^2}\}$ 
then has auto-correlation $A(s,t)=\Theta(1/\log|t-s|)$ 
for $1\ll |s-t| = \Theta(t)$. This corresponds
to $\alpha=0$, a case for which Theorem \ref{gen2} does not apply,
but Proposition \ref{decay_slowly_varying} predicts that 
$$
(\log T)^2\lesssim -\log \P(\sup_{t \in [1,T]} \{g_t\} <0)\lesssim(\log T)^3,
$$
as indeed proved in \cite{Sak}.
\end{remark}

\begin{remark}\label{rem:fbm}
As another application of Theorem \ref{gen2}, we determine the exact rate of persistence decay for stationary fractional Brownian motion of order $H\in (1/2,1)$ defined by the stochastic integral
$$
Y_H(t):=\int_{-\infty}^{t}e^{-(t-s)}dB_H(s),
$$
where $B_H(.)$ is two sided fractional Brownian motion of order $H\in (1/2,1)$. We 
refer to \cite{Unt} for a definition of stochastic integration with respect to fBM with Hurst index $H>\frac{1}{2}$.
Using \cite[(1.1)]{Unt} we find that the stationary correlation function of $Y_H(\cdot)$ is
$$
\Lambda_H(0,\tau)=e^{-\tau}+\frac{1}{\E [Y_H(0)^2]}\int_0^{\tau} e^{-(\tau-s)}R(s)ds,
$$
where 
$$
R(s):=H(2H-1)\int_0^\infty e^{-v}(v+s)^{2H-2}dv
$$
is asymptotically of order $H(2H-1)s^{2H-2}$ for $s$ large. Consequently we have that
$$
\limsup_{\tau\rightarrow\infty}\frac{\Lambda_H(t,t+\tau)}{\tau^{2H-2}}=
\frac{H(2H-1)}{\E [Y_H(0)^2]},
$$
with conditions \eqref{nonsum2} and \eqref{nonsum1} satisfied for the 
regularly varying $\rho(\tau)=\min(1,\tau^{2H-2})$. 
As for \eqref{need1}, recall that $\Lambda_H (0,\tau) {\ge 
\Lambda_{1/2}(0,\tau)} = e^{-\tau}$, hence by Slepian's lemma,
$$\E [\sup_{t\in [s,s+u]} \{Y_H(t)\} ] \le \E [\sup_{t\in [s,s+u]} \{Y_{1/2}(t)\}]
$$
for the stationary \abbr{OU} process $Y_{1/2}(\cdot)$, which 
satisfies \eqref{need1}. Consequently, so does $Y_H(\cdot)$ and 
from Theorem \ref{gen2} we conclude that 
$$
-\log \P(\sup_{[0,T]}Y_H(t)<0)=\Theta(T^{2-2H}\log T) \,.
$$
\end{remark}

Section \ref{proof-gen2} is devoted to the
proof of Theorem \ref{gen2}, with
Theorem \ref{gen2} applied in Section \ref{sec-3}
to yield part (c) of Theorem \ref{all} (and  
\cite[Theorem 1.6]{DM} utilized for deducing the 
complementary parts (a) and (b) of Theorem \ref{all}, as well as
part (e) of Corollary \ref{sak_improved}).

\noindent
\vskip 10pt

\noindent
{\bf Acknowledgment} 
This research is the outgrowth of discussions with
H. Sakagawa during a research visit of A. D.
that was funded by T. Funaki from Tokyo University. 
We are indebted to H. Sakagawa for sharing 
with us a preprint of \cite{Sak}, to J. Ding for an 
alternative proof of Theorem \ref{gen2}(b) and 
to O. Zeitouni for helpful discussions. We thank G. Schehr for bringing the references \cite{BMS,GPS} to our notice and the referees 
whose suggestions much improved this article.

\section{Proof of Theorem \ref{gen2}}\label{proof-gen2}

We first collect a few standard, well known results about 
Gaussian processes, that will be used throughout this paper.

\subsection{Preliminaries on Gaussian processes}

A key tool in our analysis is the following comparison theorem, 
known in literature as Slepian's lemma (see \cite[Theorem 2.2.1]{AT}).
\begin{thm}[Slepian's Lemma]\label{thm:slep}$~$\\
Suppose centered Gaussian processes $\{X_t\}_{t\in I}$ and $\{Y_t\}_{t\in I}$ 
are almost surely bounded on $I$. If  
$$\E X_t^2=\E Y_t^2, \quad \forall t\in I,
\qquad \E X_t X_s\le \E Y_t Y_s,\quad \forall s,t\in I \,,
$$
then for any $u\in \R$ one has
$$\P(\sup_{t\in I}X_t<u)\le \P(\sup_{t\in I}Y_t<u).$$
\end{thm}
Combining Slepian's lemma and sub-additivity, one has the following 
immediate corollary.
\begin{cor}
If $\{X_t\}_{t\ge 0}$ is a centered, stationary Gaussian process of non-negative correlation function, such that $\sup_{t\in [0,T]} X_t$ is almost surely 
finite for any $T < \infty$, then the limit
$$-\lim_{T\rightarrow\infty}\frac{1}{T}\log \P(\sup_{t\in [0,T]}X_t<0)$$
exists in $[0,\infty]$.
\end{cor}

The Sudakov-Fernique inequality (see \cite[Theorem 2.2.3]{AT}), is 
another comparison tool we use.
\begin{thm}[Sudakov-Fernique]\label{thm:sud}$~$\\
Suppose centered Gaussian processes $\{X_t\}_{t\in I}$ and $\{Y_t\}_{t\in I}$ 
are almost surely bounded on $I$. If 
$$
\E (X_t-X_s)^2\le \E (Y_t-Y_s)^2,\qquad \forall s,t\in I,
$$
then one has
$$\E[\sup_{t\in I}X_t]\le \E[\sup_{t\in I}Y_t]\,.$$
\end{thm}

We often rely on Borell-TIS inequality (see \cite[Theorem 2.1.1]{AT}) to 
provide concentration results for the supremum of Gaussian processes.
\begin{thm}[Borell-TIS]\label{thm:borell-tis}
If centered Gaussian process $\{X_t\}_{t\in I}$ is almost surely 
bounded on $I$, then $\E [\sup_{t\in I}X_t] < \infty$ and for 
$\sigma^2_I := \sup_{t\in I}\E X_t^2$ and any $u>0$, 
$$
\P(\sup_{t\in I}X_t-\E \sup_{t\in I}X_t>u)\le e^{-u^2/2 \sigma^2_I} \,.
$$
\end{thm}
We conclude with the standard formula for the distribution of 
a Gaussian process conditioned on finitely many coordinates.
\begin{thm}\label{thm:conditional}
If centered Gaussian process $\{Z_t\}_{t\ge 0}$ has covariance $A(\cdot,\cdot)$,
then for any $\ell$ distinct indices $0 \le t_1 < \cdots < t_\ell$,
conditional on $(Z_{t_i},1\le i\le \ell)$ the process $Z_t$ 
has Gaussian distribution of mean  
$$
m(t):=\sum_{i,j=1}^{\ell}\Omega(i,j)A(t,t_i)Z_{t_j}
$$ 
and covariance function
$$
\widetilde{A}(s,t):=A(s,t)-\sum_{i,j=1}^{\ell}A(s,t_i) \Omega(i,j) A(s,t_j),
$$ 
where $\Omega^{-1}$ is the $\ell$-dimensional covariance matrix of the 
centered Gaussian vector $(Z_{t_i},1\le i\le \ell)$.
\end{thm}

\subsection{Proof of Theorem \ref{gen2}}

We begin by showing the positivity of persistence probabilities 
over compact intervals for centered Gaussian processes of unit variance, 
non-negative correlation and a.s. continuous sample path.
\begin{lem}\label{clm1}
Suppose the centered Gaussian process $\{Z(t)\}$ has a.s. 
continuous sample paths, unit variance and non-negative correlation. 
Then, for any $u \in \R$ and compact interval $I$,
$$
\P(\sup_{t\in I}\{Z(t)\}<u) > 0 \,.
$$ 
\end{lem}
\begin{proof} Let $M(I):=\sup_{t\in I}\{Z(t)\}$ and suppose that $\P(M(I_1)<u)=0$ 
for some $u \in \R$ and compact interval $I_1$. Representing $I_1$ as the 
disjoint union of intervals $I^{(-)}_2$ and $I^{(+)}_2$ each of 
half the length of $I_1$, we get by the non-negativity of correlations 
and Slepian's lemma (Theorem \ref{thm:slep}) that
$$
0=\P(\, M(I_1)<u) \ge \P(\, M(I_2^{(-)})<u) \, \P(\, M(I^{(+)}_2)<u) \,.
$$
So, either $\P(M(I^{(-)}_2)<u)=0$ or $\P(M(I^{(+)}_2) <u)=0$ and 
proceeding inductively
with the sub-interval for which we have zero probability, we construct 
non-empty nested compact intervals $I_k$ of shrinking diameters
such that $\P(M(I_k)<u)=0$ for all $k$. By Cantor's intersection theorem,
$\bigcap_k I_k$ is a single non-random point $t_\star$. Thus, by the continuity of  
sample paths we get that a.s.
$$
\lim_{k \to \infty} M(I_k) = Z(t_\star)\,.
$$ 
Consequently, 
$$
0 = \lim_{k \rightarrow\infty}\P(\,M(I_k) < u) \ge \P( Z({t_\star})<u )>0 \,,
$$ 
a contradiction which rules out our hypothesis that 
$\{M(I_1)<u\}$ has zero probability.
\end{proof}

We recall some properties of positive, measurable 
slowly varying functions, that are used throughout this paper.
\begin{remark}\label{bgtslow}
For any {$L \in \cR_0$ (namely, positive, measurable, slowly varying function 
on $\R_+$),} the convergence of $\frac{L(\lambda t)}{L(t)}$ to $1$ 
is uniform over $\lambda$ in a compact subset of $(0,\infty)$
(see \cite[Theorem 1.2.1]{BGT}). Further, by the representation theorem 
(see \cite[Theorem 1.3.1]{BGT}),
there exists then  
$\widetilde{L} \in \cR_0$ such that 
$$
\lim_{x\rightarrow\infty}\frac{\widetilde{L}(x)}{ L(x)}=1,
$$ 
and $x \mapsto x^\eta \widetilde{L}(x)$ is eventually increasing (decreasing) 
if $\eta>0$ ($\eta<0$ resp.). That is, up to a universal constant {factor
(that depend on $L(\cdot)$), the function $x^\eta L(x)$ may be assumed}
eventually increasing (decreasing) if $\eta>0$ ($\eta<0$ resp.).		
\end{remark}

W.l.o.g. we assume throughout that $\{Z(\cdot)\}$ 
has been re-scaled so that $\E[Z(t)^2]=1$ for all 
$t \in \R_+$ and state next three auxiliary lemmas 
which are needed for proving Theorem \ref{gen2}
(while deferring the proof of these lemmas 
to the end of the section).
\begin{lem}\label{clm2}
For $\alpha\ge 0$ and $\rho\in\mathcal{R}_\alpha$, let 
${L}_\rho(x)=x^\alpha \rho(x) \in \cR_0$ {for which we further assume  
the eventual monotonicity properties of Remark \ref{bgtslow}.}
	\begin{enumerate}[(a)]\item
		If $0\le \alpha<1$,
		\begin{align}\label{eq:slow_estimate<1}
		\lim\limits_{b\rightarrow\infty}\sup_{a\in [0,b)}\Big|\frac{I_\rho(b)-I_\rho(a)}{L_\rho(b)(b^{1-\alpha}-a^{1-\alpha})}-\frac{1}{1-\alpha}\Big|=0.
		\end{align}
		
		\item
		If $\alpha>1$,
		\begin{align}\label{eq:slow_estimate>1}
		\lim\limits_{b\rightarrow\infty}\sup_{a\in (b,\infty)}\Big|\frac{I_\rho(a)-I_\rho(b)}{L_\rho(b)(b^{1-\alpha}-a^{1-\alpha})}-\frac{1}{\alpha-1}\Big|=0.
		\end{align}

	\end{enumerate}
\end{lem}
\begin{lem}\label{slow}
Suppose $\rho \in \cR_\alpha$ for $\alpha  \ge 0$.
\begin{enumerate}[(a)] 
\item
The function $I_\rho(\cdot)$ is a regularly varying function of order $(1-\alpha)_+$
and
\begin{equation}\label{eq:lem-2.1_lower_bound} 
 \limsup_{n,M\rightarrow\infty}\frac{M\sum_{\ell=1}^{n}\rho(\ell M)}{I_\rho(nM)}<\infty.
\end{equation}
If  $I_\rho(\infty)<\infty$, then we have the stronger conclusion
\begin{equation}\label{eq:lem-2.1_lower_bound_new} 
 \limsup_{M\rightarrow\infty}M\sum_{\ell=1}^{\infty}\rho(\ell M)=0.
\end{equation}
\item Suppose $\alpha\in [0,1]$, {with $\rho(x) \to 0$ when $\alpha=0$}
and $I_\rho(\infty)=\infty$ when $\alpha=1$. Then, 
fixing $\mu>0$ we have for $M:=\mu I_\rho(T)$ 
that 
\begin{equation}\label{eq:lem-2.1} 
 \lim_{T\rightarrow\infty} \sum_{\ell=1}^{\lceil T/M\rceil}\rho(\ell M) =\frac{1}{\mu}\,.
\end{equation}
\end{enumerate}
\end{lem}

\begin{lem}\label{easy}
	If the auto-correlation $A(\cdot,\cdot)$ of a centered Gaussian process $\{Z(\cdot)\}$
satisfies \eqref{nonsum2} for some 
$\rho\in \mathcal{R}_\alpha$ and $\alpha\in (0,1]$.
Then, there exist ${\eta},\delta>0$ such that 
\begin{equation}\label{eq:lem-2.2}
\lim\limits_{M\rightarrow\infty}\frac{1}{\log M} 
\sup_{s \ge M/\eta}
\log\P(\sup_{t\in[s,s+M]}\{Z(t)\}<\sqrt{\delta\log M})=-\infty.
\end{equation}
\end{lem}

\begin{proof}[Proof of Theorem \ref{gen2}]
Throughout the proof, all constants implied by the notation 
$\lesssim$ depend only on the 
function $A(\cdot,\cdot)$.

\noindent
(a).
By (\ref{nonsum2}) there exist $\widetilde{\eta}>0$ small, $T_\star$ finite 
and $r = 1-\widetilde{\eta} \in (0,1)$ such that 
\begin{align}\label{u00}
\sup_{t \in [rT,T-\tau]} A(t,t+\tau) \lesssim\rho(\tau)\,, \qquad \forall \tau \ge 0\,,\;\; 
T \ge T_\star \,.
\end{align}
For some large universal constant $\lambda < \infty$ to be chosen in 
the sequel, we set $M=M(T):=\lambda I_\rho(T)$ and 
$n=n(T):=\lfloor \frac{(1-r) T}{2M}\rfloor$, 
both of which diverge with $T \to \infty$ due to 
our assumptions on $\rho(\cdot)$. We then 
consider the following subset of $[rT,T]$,  
$$
\cJ:=\bigcup\limits_{\ell=1}^n \cJ_{2\ell}\,.
$$ 
That is, $\cJ$ is the union of every other sub-interval $\cJ_\ell:=[s_\ell,s_{\ell+1}]$, 
where $s_\ell:= rT + (\ell-1) M$ for $\ell \ge 1$, 
and $n(T)$ is the largest 
$\ell \in \N$ such that $s_{2\ell+1} \le T$. 
With $\cJ \subset [0,T]$, we
trivially have that 
\begin{equation}\label{eq:trivial-ubd}
\P(\,\sup_{t\in[0,T]}\,\{Z(t)\}<0)\le \P(\,\sup_{t\in \cJ}\,\{Z(t)\}\,<0)\,.
\end{equation}
For $t\in \cJ$ let $J(t)=\ell$ when $t\in \cJ_{2\ell}$ 
for some $\ell \in \{1,2,\ldots, n\}$ noting that for all $s,t \in \cJ$, 
\begin{align}\label{u3}
A(s,t)\le \frac{1}{2}A(s,t) 1_{\{J(s)=J(t)\}}+\frac{1}{2}B(J(s),J(t)),
\end{align}
where 
\begin{equation}\label{eq:B-def}
B(i,j):=2\sup\{A(s,t): \,s \in \cJ_{2i},\, t \in \cJ_{2j} \}\,,  \text{ if }i\ne j,
\qquad B(i,i)=1. 
\end{equation}
If $s,t \in \cJ$ with $J(s) \ne J(t)$, then clearly 
\begin{equation}\label{eq:time-bd}
M |J(s)-J(t)| \le |s-t| \le  3 M |J(s)-J(t)|\,. 
\end{equation}
Since 
$M = M(T) \to \infty$, 
we have
from \eqref{u00}, \eqref{eq:B-def} and \eqref{eq:time-bd}, that with $\rho(\cdot)$
regularly varying, 
\begin{equation}\label{eq:A-unv-bd}
B(i,j) 
\lesssim \sup_{x \in [M|i-j|,3M|i-j|]} \rho(x)
\lesssim \rho( M |i-j|)\,,
\end{equation}
uniformly in $i \ne j$ and for all $M$ large enough. 
We thus deduce by 
\eqref{eq:A-unv-bd} and \eqref{eq:lem-2.1_lower_bound} that
\begin{align}\label{eq:xi-bd}
\xi(T) := \sup_{1 \le j \le n} \; \big\{ \sum\limits_{i=1, i\neq j}^n \, B(i,j) \big\} 
\lesssim \sum\limits_{\ell=1}^{n} \rho(\ell M) \lesssim \frac{I_\rho(n M)}{M} 
\le \frac{1}{\lambda}
\end{align}
where the right-most inequality results from having 
chosen $M=\lambda I_\rho(T)$ and $n M \le T$ (so 
$I_\rho(nM) \le I_\rho(T)$). The universal constant on the \abbr{rhs} of 
\eqref{eq:xi-bd} is independent of $\lambda$, hence 
there exist $\lambda=\lambda_1$ and $T_{\star\star}$ finite, such that  
$\xi(T) \le 1/2$ for all $T \ge T_{\star\star}$.
Using hereafter $\lambda=\lambda_1$ for the remainder of part (a) and the fact that $B(i,i)=1$, it follows  
by the Gershgorin circle theorem and the interlacing property of eigenvalues, that for any $T \ge T_{\star\star}$, 
the principal sub-matrices of the 
symmetric $n$-dimensional matrix ${\bf B}=\{B(i,j)\}$   
have all their eigenvalues within $[1/2,3/2]$. In particular, ${\bf B}$ is positive definite, and  
with $\{X_\ell\}_{\ell=1}^n$ denoting the centered 
Gaussian random vector of covariance matrix ${\bf B}$, 
upon applying the argument in 
\cite[display following (2.5)]{DM} 
for principal sub-matrices of ${\bf B}$, 
we get that 
for any $L>0$ and all $1 \le i_1 < \cdots < i_k \le n$,
\begin{align}\label{u5}
\P(\sup_{l=1}^k \{ X_{i_l} \} <-\sqrt{L}) \le 
3^{k/2}\P(X_1 > \sqrt{2L/3})^k \,.
\end{align}
Next, we denote by $\{\overline{Z}(t), t \in \cJ\}$ the centered Gaussian 
process which has the same law as $\{Z(t)\}$ when restricted to each 
sub-interval $\cJ_{2\ell}$, while being independent across different 
sub-intervals and independent of the random vector $\{X_\ell\}_{\ell=1}^{n}$.
Then by (\ref{u3}), upon applying Slepian's lemma 
we get that for all $L>0$ and $T\ge \max(T_\star,T_{\star\star})$,
\begin{align}\label{eq:bd-supJ}
\P(\sup_{t \in \cJ}\{Z(t)\}<0)\le \E \Big[ \prod\limits_{\ell=1}^{n(T)}\big(
\P(\sup_{t\in\cJ_{2\ell}} \{\overline{Z}(t)\}<\sqrt{L})+1_{\{X_\ell<-\sqrt{L}\}} \big) \Big]
\end{align}
(c.f. \cite[(2.4)]{DM} for a more detailed version of 
this argument). Utilizing now (\ref{u5}), we deduce 
from \eqref{eq:bd-supJ} in a similar manner as 
the derivation of \cite[(2.6)]{DM} that for all $L>0$,
\begin{align}
\P(\,\sup_{t \in \cJ}\{Z(t)\}<0\,) &\le
\Big[ f(\delta,T)+\sqrt{3} \P(X_1>\sqrt{2L/3}) \Big]^{n(T)} \nonumber \\ 
&\le 2^{n(T)} \max\Big[f(\delta,T),\sqrt{3} \P(X_1>\sqrt{2L/3})\Big]^{n(T)}
\,,
\label{eq:new-rev1}
\end{align}
where 
$$
f(\delta,T) := 
\sup_{\ell=1}^{n(T)} \, \P(\, \sup_{t\in\cJ_{2\ell}} \, \{Z(t)\} < \sqrt{L} \,) \,.
$$
Moreover, setting $L=L(T):=\delta \log I_\rho(T)$,
upon considering \eqref{eq:lem-2.2} for the
intervals $\cJ_{2\ell}$ within $[rT,T]$,
of length $M(T)=\lambda_1 I_\rho(T)$ each,
we can choose $\delta>0$ 
small enough so that  
 \begin{align}\label{u6}
\lim\limits_{T\rightarrow\infty}\, \frac{\log f(\delta,T)}{\log I_\rho(T)}=-\infty
\,.
\end{align}
Since $n(T) L(T) = \frac{(1-r) \delta}{2 \lambda_1} a_\rho(T)$,
considering $- a_\rho(T)^{-1} \log$ of both sides of 
\eqref{eq:new-rev1}, we deduce from the usual 
tail estimates for the $N(0,1)$ law of $X_1$, that 
$$
- \limsup\limits_{T\rightarrow\infty}\frac{1}{a_\rho(T)}\log \P(\sup_{t\in \cJ}\{Z(t)\}<0)\ge 
\frac{(1-r)\delta}{6\lambda_1} > 0 
$$
(with a negligible contribution of $f(\delta,T)$ due to \eqref{u6}). Combined
with \eqref{eq:trivial-ubd} this yields the stated upper bound \eqref{eq:ubd-pers}.

\bigskip
\noindent
(b) {\bf Step I.} We first show that suffices for \eqref{eq:theta-pers} to 
have for some $r \in (0,1)$ and finite $C_1$, $T_1$,  
\begin{align}\label{u1}
\P(\sup_{t\in[rT,T]}\{Z(t)\}<0)\ge e^{-C_1 a_\rho(T)} 
\qquad \quad \forall\, T \ge T_1 \,.
\end{align}
To this effect, for $T\ge T_1$ set $m=m(T):=\big\lceil \frac{\log (T/T_1)}{\log(1/r)}\big\rceil$ and
\begin{equation}\label{eq:wh-a-rho}
\widehat{a}_\rho(T):=\sum_{i=1}^m a_\rho(T_i) \,, \quad T_i := r^{1-i} T_1,\;\; i \ge 0 \,,
\end{equation}
so that $T_m \in [T,T/r]$. With $A(\cdot,\cdot)$ non-negative, by 
Slepian's lemma and \eqref{u1} we have
\begin{align}
\P(\sup_{t\in[0,T]}\{Z(t)\}<0) &\ge \P(\sup_{t\in[0,T_0]}\{Z(t)\}<0)\prod\limits_{i=1}^{m}
\P(\sup_{t\in[T_{i-1},T_i]}\{Z(t)\}<0)\nonumber \\
&\ge \P(\sup_{t\in[0,T_0]}\{Z(t)\}<0) e^{-C_1
\widehat{a}_\rho(T)}\,.
\label{eq:lbd-no-r}
\end{align}
Next, by Lemma \ref{clm1} the 
event $\{\sup_{t\in[0,T_0]}\{Z(t)\}<0\}$ has 
positive probability, so 
considering the $\limsup$ as $T \to \infty$ of 
$- a_\rho(T)^{-1} \log$ of both sides of the 
preceding inequality, we get \eqref{eq:theta-pers} 
upon showing that  
\begin{align}\label{eq:arho-equiv}
\widehat{a}_\rho(T) \lesssim a_\rho(T) \,. 
\end{align}
To this end, recall by Lemma \ref{slow} that $t \mapsto I_\rho(t)$ 
is regularly varying of order $1-\alpha$, hence
$a_\rho(t)=t\log I_\rho(t)/I_\rho(t)$ is regularly 
varying of order $\alpha>0$. Thus, 
there exists $K \ge 1$ finite such that 
$a_\rho(T_{i-1})\le r^{\alpha/2} a_\rho(T_i)$
for all $i \ge K$, from which we deduce that 
$$
\widehat{a}_\rho(T) 
\le \sum\limits_{i=1}^K a_\rho(T_i)
+ a_\rho(T_m) \sum_{l=0}^{\infty} r^{\alpha l/2} \,.
$$
The first sum on the \abbr{rhs} is finite and independent 
of $T$. Further, with $\alpha>0$ and $r \in (0,1)$, 
the same applies for the second sum there. 
Lastly, since $a_\rho(\cdot)$ is regularly varying 
of positive order, $a_\rho(T_m) \lesssim a_\rho(T)$,
yielding \eqref{eq:arho-equiv} and thereby 
\eqref{eq:theta-pers}. 

\smallskip
\noindent
{\bf Step II.} We proceed to verify \eqref{u1} for 
$r=1 - \eta \in (0,1)$ and $\eta$ small enough so that 
by \eqref{nonsum1} in addition to \eqref{u00}, we further 
have for all $T$ large enough 
\begin{equation}\label{l00}
\rho(\tau) \lesssim \inf_{t \in [rT,T-\tau]} A(t,t+\tau) \,, \qquad \forall \tau \ge 0 \,
\,.  
\end{equation} 
Then, for such $T$ large, set 
$M=M(T)=\lambda I_\rho(T)$ for some finite  
$\lambda$ to be chosen in the sequel, and cover the interval
$[rT,T]$ by 
$$
\cJ := \bigcup_{\ell=1}^{3n'} \overline{\cJ_\ell}
$$ 
for the corresponding open sub-intervals 
$\cJ_\ell=(s_{\ell},s_{\ell+1})$, with 
$s_\ell=rT+(\ell-1)M$ for $\ell=1,\ldots,n'$ 
and $n'=n'(T):=\lceil \frac{(1-r)T}{3M}\rceil$ 
the smallest integer for which $[rT,T] \subseteq \cJ$.
An application of Slepian's lemma gives
\begin{align*}
\P(\sup_{t\in[rT,T]}\{Z(t)\}<0) \ge \prod_{i=0}^2\P(\sup_{t\in \cJ^i}\{Z(t)<0),\quad 
\cJ^i:=\bigcup_{\ell=1}^{n'} \overline{\cJ}_{3\ell-i} 
\end{align*}
We will show that for $i=0$,
\begin{align}\label{claim:lower_bound}
\liminf_{T\rightarrow\infty}\frac{1}{a_\rho(T)}\log\P(\sup_{t\in \cJ^i}\{Z(t)\}<0)
> - \infty 
\end{align}
and with the same reasoning applicable for $i=1,2$, the bound \eqref{u1} follows for suitably chosen $C_1,T_1$ finite.
Turning to show \eqref{claim:lower_bound}, we take $L=L(T):=\beta\log I_\rho(T)$ 
for some $\beta=\beta(\lambda)$ finite to be determined later, and 
get a lower bound by enforcing the event
\begin{align*}
\Gamma   &:= \{\; -\sqrt{L}< Z(s_{3\ell-1}) < -(3/4) \sqrt{L} \,, 
\;\quad \ell=1,\ldots,n'\} \,, 
\end{align*}
which is measurable with respect to the $\sigma$-algebra
$\cF := \sigma(Z(s_{3\ell-1}),1 \le \ell \le n')$.
Indeed,
\begin{align}
\P(\sup_{\ell=1}^{n'}
\sup_{t\in \overline{\cJ}_{3\ell}}\{Z(t)\}<0)
 \ge\E\Big[ 
 \P{\Big(\bigcap_{\ell=1}^{n'}\sup_{t\in \overline{\cJ}_{3\ell}}\, \{Z (t) \} <0} \
|\cF \Big)\; 1_{\Gamma} \, \Big] 
\label{not_so_clever}
\end{align}
and proceeding to bound the \abbr{RHS} of (\ref{not_so_clever}), let 
$m(t)$ denote the conditional mean of $Z(t)$ given $\cF$. 
We claim that for some choice of $\lambda=\lambda_2$ and 
$C_2=C_2(\lambda)>0$ one has that 
\begin{align}\label{mean_estimate_final}
\Gamma \qquad \Longrightarrow \qquad \qquad \qquad 
-\sqrt{C_2 L} & \ge \sup_{t \in \cJ^0} \, \{ m(t) \} \,,
\\
{\sf Cov}(Z(u),Z(v)|\cF) & \ge 0 \,.
\label{eq:cov_positive}
\end{align} 
We next complete the proof of \eqref{claim:lower_bound} assuming 
both \eqref{mean_estimate_final} and \eqref{eq:cov_positive} hold
(deferring to Step III the proof of the latter estimates). Indeed,
by Slepian's lemma it then follows that  
 \begin{align}\label{not_so_clever+slepian}
 \P\Big(\bigcap_{\ell=1}^{n'}\sup_{t\in \overline{\cJ}_{3\ell}}\{Z(t)\}<0|\cF\Big)
 1_{\Gamma} &\ge 
\prod_{\ell=1}^{n'}\P\Big(\sup_{t\in \overline{\cJ}_{3\ell}}\{Z(t)\}< 0 |\cF\Big)
1_{\Gamma} \nonumber \\ 
&\ge 
\prod_{\ell=1}^{n'}\P\Big(\sup_{t\in \overline{\cJ}_{3\ell}}\{Z(t)-m(t)\}<\sqrt{C_2 L}\,|\cF\Big) 1_{\Gamma} \,.
 \end{align}
Conditional on $\cF$, the centered 
normal random variable $Y_{u,v}:=Z(u)-m(u)-Z(v)+m(v)$ has
variance $\E[Y_{u,v}^2| \cF] \le \E[(Z(u)-Z(v))^2]$.
Thus, by the Sudakov-Fernique inequality (Theorem \ref{thm:sud}), for any $s,u>0$,
a.s. 
$$
\E \big[ \sup_{t\in[s,s+u]}\{Z(t)-m(t)\}|\cF \big] \le 
\E \big[ \sup_{t\in [s,s+u]}\{Z(t)\} \big] \,.
$$
Thus, from \eqref{need1} there exist $u_0>0$ and $K<\infty$
(independent of $T$ and $M$), such that a.s.
$$
\sup_{s\ge 0}\E[ \sup_{t\in[s,s+u_0]}\{Z(t)-m(t)\}|\cF] \le K \,.
$$
Upon covering $\cJ_{3\ell}$ by intervals of length $u_0$, in each of 
which we apply the Borell-TIS inequality (Theorem \ref{thm:borell-tis}),
for the conditional Gaussian centered process $\{Z(t)-m(t)\}$ of 
maximal variance one, we get by a union bound that a.s.
\begin{align*}
\P(\sup_{t\in \cJ_{3\ell}}\{Z(t)-m(t)\} \ge \sqrt{C_2 L}\,|\cF) 
&\le \lceil M/u_0 \rceil \sup_{s \ge 0}\P(\sup_{t\in[s,s+u_0]}\{Z(t)-m(t)\} \ge \sqrt{C_2 L}\,|\cF)\\
&\le \frac{2 M}{u_0} \exp\big\{-\frac{1}{2} (\sqrt{C_2 L}-K)^2\big\} \,.
\end{align*}
With $M = \lambda_2 I_\rho(T)$, $L=\beta \log I_\rho(T)$ and $I_\rho(T) \uparrow \infty$,
upon taking $\beta>2/C_2$ the \abbr{rhs} is bounded by $1/2$ for all $T$ large enough.
In this case, we deduce from \eqref{not_so_clever} and \eqref{not_so_clever+slepian} 
that
\begin{align*}  
 \P(\sup_{t\in \cJ^0} \{Z(t)\}<0) \ge 2^{-n'} \P(\Gamma)
& = 2^{-n'} \P\Big(\bigcap_{\ell=1}^{n'} 
\{\frac{3}{4} \sqrt{L} < Z(s_{3\ell-1}) < \sqrt{L} \} \Big) \\
&\ge 12^{-n'/2} \P\big(\frac{3}{4} \sqrt{2L} < Z < \sqrt{2L}\big)^{n'},
\end{align*}
where noting that by \eqref{mean_estimate} all eigenvalues 
of the covariance matrix ${\bf I} + \bDel$ are within $[1/2,3/2]$, 
the last inequality follows by the same argument employed in 
\cite[display following (2.5)]{DM}.
Considering the limit as $T\rightarrow\infty$ 
of $-a_\rho(T)^{-1} \log$ of both sides, results with 
\begin{align*}
\liminf_{T\rightarrow\infty}\frac{1}{a_\rho(T)}\log \P(\sup_{t\in\cJ^0}\{Z(t)\}<0)\ge -\frac{(1-r)\beta}{16\lambda_2}\,,
\end{align*}
thereby establishing \eqref{claim:lower_bound}, and consequently \eqref{u1}.

\smallskip
\noindent
{\bf Step III.} It remains only to establish
\eqref{mean_estimate_final} and \eqref{eq:cov_positive}.
To this end, setting $\bDel$ the $n'$-dimensional matrix of non-negative entries
$$
\Delta(\ell,\ell'):={\sf Cov}(Z(s_{3\ell-1}),Z(s_{3\ell'-1}))\text{ for }1\le \ell\ne \ell'\le n', \quad \Delta(\ell,\ell):=0,
$$ 
we claim that there exists $\lambda_2$ such that for all $\lambda\ge \lambda_2$ the following 
estimates hold simultaneously:
\begin{align}
\max\Big\{\max_{\ell=1}^{n'}\sum_{\ell'=1}^{n'}\Delta(\ell,\ell'),\sup_{t\in \cJ^0}\sum_{\ell=1}^{n'}A(t,s_{3\ell-1})\Big\}\leq&\frac{1}{2}, 
\label{mean_estimate}\\
\max_{a, b =1, a \ne b}^{n'}\sup_{v\in \overline{\cJ}_{3a}, u \in \overline{\cJ}_{3b}}\, \frac{1}{A(u,v)}\sum_{\ell=1}^{n'}A(u,s_{3\ell-1})A(v,s_{3\ell-1})\le &\frac{1}{2},
\label{cov_estimate_blocks}\\
\max_{a,b=1}^{n'}\sup_{v\in\overline{\cJ}_{3a}}\frac{1}{A(v,s_{3b-1})}\sum_{\ell=1}^{n'}\Delta(b,\ell) A(v,s_{3\ell-1}) \leq&\frac{1}{2}\,.
\label{cov_estimate_endpoints}
\end{align}
While deferring the proof of \eqref{mean_estimate}--\eqref{cov_estimate_endpoints} to Step IV, we fix hereafter $\lambda=\lambda_2$
and rely on these bounds to establish \eqref{claim:lower_bound}. Indeed, setting the 
vectors
$$
\by(t)
:=[A(t,s_{3\ell-1}),1\le \ell\le n'],\qquad \bz:=[Z(s_{3\ell-1}),1\le \ell\le n']$$ and utilizing Theorem \ref{thm:conditional},
the conditional mean of $Z(t)$ given $\cF$ is 
\begin{align*}
m(t)= &\langle \by(t), ({\bf I}+\bDel)^{-1}\bz \rangle
 =\sum_{k=0}^\infty \langle \by(t) , \Delta^{2k}({\bf I}-\bDel) \bz \rangle
 \end{align*}
where the expansion as a power series requires that the operator norm of $\bDel$ is less than $1$. To verify this, recall that the operator norm of a symmetric matrix is bounded by the maximum row sum, which coupled with
\eqref{mean_estimate} gives 
$\|\bDel {\bf 1}\|_\infty \le 1/2$.  In view of the latter bound on $\bDel$,
the event $\Gamma$ implies that 
\begin{align*}
\langle \by(t), \bDel^{2k}({\bf I}-\bDel) \bz \rangle \le -\sqrt{L}
\langle \by(t) , \bDel^{2k}\Big[\frac{3}{4}{\bf 1}-\bDel {\bf 1}\Big] \rangle \le 
-\frac{\sqrt{L}}{4} \langle \by(t), \bDel^{2k}{\bf 1} \rangle \,.
 \end{align*}
In particular, this is negative for any $k\ge 1$, hence under $\Gamma$,
\begin{align}\label{eq:new_added_1}
m(t) \le \langle \by(t), ({\bf I}-\bDel) \bz \rangle 
\le -\frac{\sqrt{L}}{4} \langle \by(t), {\bf 1} \rangle 
= - \frac{\sqrt{L}}{4}
\sum_{\ell=1}^{n'}A(t,s_{3\ell-1}) \,.
\end{align}
Further, recall that for any $t \in \cJ^0$ the elements of  
$\{|t-s_{3\ell-1}|/M\}$ are of 
the form $\{\theta + 3 \Z \}$, for 
some $\theta=\theta(t) \in [1,2]$. Hence, with $\rho(\cdot)$ and $I_\rho(\cdot)$
regularly varying, by \eqref{l00} and \eqref{eq:lem-2.1} (for 
$3M/I_\rho((1-r)T) \to \mu = c \lambda_2$), 
\begin{equation}\label{eq:ubd-At}
\inf_{t\in \cJ^0} \, \sum_{\ell=1}^{n'}A(t,s_{3l-1})\gtrsim \, \inf_{\theta \in [1,2]} 
\{ \sum_{k=0}^{n'-1}\rho((\theta + 3k) M) \} \gtrsim \frac{1}{\lambda_2}\,.
\end{equation}
Combining \eqref{eq:new_added_1} and \eqref{eq:ubd-At}, we get 
the existence of  $C_2$ independent of $T$ and $L$ for which\eqref{mean_estimate_final} holds.
Proceeding to bound the covariance of the conditional process across blocks, we set
$$
a_k(u,v):=\sum_{\ell,\ell'=1}^{n'}A(u,s_{3\ell-1})\bDel^{k}(\ell,\ell')A(v,s_{3\ell'-1}),
$$
for $u\in \cJ_{3j},v\in\cJ_{3j'}$ with $1\le j\neq j' \le n'$, and use Theorem \ref{thm:conditional} to note that
\begin{align}
 {\sf Cov}(Z(u),Z(v)|\cF)=A(u,v)-\sum_{k=0}^\infty (-1)^ka_k(u,v)
\ge A(u,v)-\sum_{k=0}^\infty a_{2k}(u,v),
\label{cov_estimate_2}
\end{align}
where $\bDel^0:={\bf I}$. Note that by \eqref{cov_estimate_endpoints}, for any $k \ge 1$,
\begin{align*}
a_k(u,v)
      =&\sum_{\ell,\ell''=1}^{n'}A(u,s_{3\ell-1})\bDel^{k-1}(\ell,\ell'')\sum_{\ell'=1}^{n'}\Delta(\ell'',\ell')A(v,s_{3\ell'-1})\\
      \le &\frac{1}{2}\sum_{\ell,\ell'=1}^{n'}A(u,s_{3\ell-1})\bDel^{k-1}(\ell,\ell'')A(v,s_{3\ell''-1})=\frac{1}{2}a_{k-1}(u,v) 
\end{align*}
and consequently, for any $k \ge 0$, by \eqref{cov_estimate_blocks},
\begin{align}\label{cov_estimate_3}
a_k(u,v)\le \Big(\frac{1}{2}\Big)^{k}a_0(u,v)\le \Big(\frac{1}{2}\Big)^{k+1}A(u,v) \,.
\end{align} 
Combining \eqref{cov_estimate_2} and \eqref{cov_estimate_3} we deduce that 
\begin{align*}
{\sf Cov}(Z(u),Z(v)|\cF)\ge A(u,v)
\Big[1-\sum_{k=0}^\infty \Big(\frac{1}{2}\Big)^{2k+1}\Big] \ge 0 \,.
\end{align*}
thus verifying \eqref{eq:cov_positive} as well.
	
	
\smallskip
\noindent
{\bf Step IV.} In proving \eqref{mean_estimate}-\eqref{cov_estimate_endpoints} we repeatedly use properties 
of regularly varying functions, and in particular, having 
$\alpha>0$, assume hereafter \abbr{wlog} that $\rho(\cdot)$ is 
eventually non-increasing (see Remark \ref{bgtslow}). Starting 
with \eqref{mean_estimate}, note that by the same argument 
used for deriving \eqref{eq:ubd-At},
\begin{align*}
\sup_{t\in \cJ^0} \, \big\{ \sum_{\ell=1}^{n'}A(t,s_{3\ell-1}) \big\}
\, \lesssim \sup_{\theta \in [1,2]} \, \{
\sum_{k=0}^{n'-1}\rho((\theta+3k)M) \} \, \lesssim \frac{1}{\lambda} \,.
\end{align*}
The same calculation shows that
$$
\max_{\ell=1}^{n'}\sum_{\ell'=1}^{n'}\Delta(\ell,\ell')
\lesssim \sum_{k=1}^{n'}\rho(3 k M) \lesssim \frac{1}{\lambda} \,,
$$
so choosing $\lambda$ large enough guarantees that \eqref{mean_estimate} holds.
Next, in view of \eqref{u00}, \eqref{l00} and 
having $\rho(\cdot)$ regularly varying and eventually non-increasing, the \abbr{lhs}
of \eqref{cov_estimate_blocks} and \eqref{cov_estimate_endpoints} are both
bounded up to a universal constant multiplicative factor, by 
$$
\rho(M) + \max_{1 \le a \le b \le n'} \{ R_{[1,n']} \} \,,
$$ 
where setting $I_1=[1,a-1]$, $I_2=[a+1,(a+b)/2]$, $I_3=[(a+b)/2,b-1]$, $I_4=[b+1,n']$,
\begin{align*}
R_{[1,n']} := & 
\sum_{\ell \neq a,b}^{n'} \,
\frac{\rho(|s_\ell-s_a|)\rho(|s_\ell-s_b|)}{\rho(|s_a-s_b|)} 
= \sum_{i=1}^4 R_{I_i} \,,
\end{align*}
and $R_{I_i}$ corresponds to the sum over $\ell \in I_i$. 
It thus suffices to show that $\max_{i,a,b} R_{I_i} \lesssim 1/\lambda$ 
(so choosing $\lambda$ large 
enough guarantees that also \eqref{cov_estimate_blocks} and 
\eqref{cov_estimate_endpoints} hold). Now with 
$\rho(\cdot)$ eventually non-increasing, we have that for $M$ large enough and 
all $a \le b$, 
$$
R_{I_1} \le \sum_{\ell=1}^{a-1} \rho(s_a-s_\ell)\,, \qquad
R_{I_4} \le \sum_{\ell=b+1}^{n'} \rho(s_\ell-s_b) \,.
$$ 
Further, $I_2$ and $I_3$ are empty unless $a<b$, in which case
$s_a^b:=s_b-s_a \ge M$ and
\begin{align*}
R_{I_2} &=  \sum_{\ell= a+1}^{(a+b)/2}
\frac{\rho(s_b-s_\ell)}{\rho(s_a^b)}\rho(s_\ell-s_a) 
\le 
\sup_{\theta \in [1/2,1]} \Big\{ \frac{\rho(\theta s^b_a)}{\rho(s^b_a)} \Big\} 
\sum_{\ell=a+1}^{(a+b)/2} \rho(s_\ell-s_a) \le C \sum_{\ell=a+1}^{(a+b)/2} 
\rho(s_\ell-s_a) \,,
\end{align*}	        
while by the same reasoning also
$$
R_{I_3} = \sum_{\ell>(a+b)/2}^{b-1}
\frac{\rho(s_\ell-s_a)}{\rho(s_a^b)}\rho(s_b-s_\ell)
\le C \sum_{\ell>(a+b)/2}^{b-1} \rho(s_b - s_\ell) \,.
$$
Combining the latter four bounds, we conclude that 
$$
\max_{i,a,b} \{R_{I_i}\} \le C \sum_{\ell=1}^{n'} \rho(\ell M) \lesssim \frac{1}{\lambda}
$$
as claimed.
\end{proof}

\begin{proof}[Proof of Proposition \ref{decay_slowly_varying}] 
The bulk of the proof of Theorem \ref{gen2} dealt with $\sup_{t \in [rT,T]} \{Z(t)\}$
for some fixed $r \in (0,1)$. This part of the proof applies even for $\rho \in \cR_0$,
under our extra assumptions that $\rho(\cdot)$ is eventually non-increasing and 
decays to $0$ at $\infty$. Thus, for some $r\in (0,1)$
\begin{align}\label{eq:r-bd}
-\infty < &\liminf_{T\rightarrow\infty}\frac{1}{a_\rho(T)}\log \P (\sup_{t \in [rT,T]}\{Z(t)\}<0) \nonumber \\
&\le \limsup_{T\rightarrow\infty}\frac{1}{a_\rho(T)}\log \P (\sup_{t \in [rT,T]}\{Z(t)\}<0)<0 \,,
\end{align}
from which the \abbr{lhs} of \eqref{eq:decay_slowly_varying} trivially follows. 
As for the \abbr{rhs} of \eqref{eq:decay_slowly_varying}, the derivation of 
\eqref{eq:lbd-no-r} remains valid here, leading to 
$$
\liminf_{T\rightarrow\infty}\frac{1}{\wh{a}_\rho(T)}\log P(\sup_{t\in [0,T]}\{Z(t)\}<0)
>-\infty,
$$
for $\wh{a}_\rho(T)$ of \eqref{eq:wh-a-rho}. 
Since $T\mapsto \rho(T)$ is non-decreasing, the map $T \mapsto a_\rho(T)$ is differentiable a.e., with
$$
T I_\rho(T) \frac{d \log a_\rho(T)}{dT} \ge \int_0^T (\rho(x)-\rho(T)) dx \,.
$$
The \abbr{rhs} is eventually non-negative due to 
our assumption that the positive and eventually non-increasing $\rho(T)$ decreases to zero 
as $T \to \infty$. Thus, the slowly varying 
$a_\rho(\cdot)$ is eventually non-decreasing, resulting 
for large enough $T_1$ with 
$$
\wh{a}_\rho(T) = \sum_{i=1}^m a_\rho(T_i) \le m a_\rho(T_m) \lesssim a_\rho(T)\log T\,,
$$
thereby completing the proof.
\end{proof}

\begin{proof}[Proof of Lemma \ref{clm2}] (a). Recall that 
$I_\rho(b)-I_\rho(a)=\int_a^b x^{-\alpha} L_\rho(x) dx$. 
Let $\delta:=a/b$
and $\mu_\alpha$ denote the probability measure on $[0,1]$ 
of density $(1-\alpha) y^{-\alpha}$.
The change of variable $x=yb$ transforms 
our claim \eqref{eq:slow_estimate<1} to 
\begin{equation}\label{eq:Fb-limit}
\lim\limits_{b\rightarrow\infty} \sup_{\delta \in [0,1)}\Big|
\frac{\int_{\delta}^1 \big(\frac{L_\rho(y b)}{L_\rho(b)} - 1 \big) d\mu_\alpha(y)}
{\mu_\alpha([\delta,1])} \Big|= 0 \,.
\end{equation}
For bounded below $\delta>0$ this 
follows from the uniformity 
of the convergence $L_\rho(y b)/L_\rho(b) \to 1$,
w.r.t. $y$ in a compact subset of $(0,1]$
(see Remark \ref{bgtslow}).
Further, $\mu_\alpha([0,\delta])=\delta^{1-\alpha} \to 0$ as
$\delta \to 0$, so fixing $0<\eta<1-\alpha$, it suffices 
to show that for some $b_0$ and $\kappa$ finite, 
all $b \ge b_0$ and any $\delta \in (0,1]$,
\begin{equation}\label{eq:small-y}
\int_0^\delta \frac{L_\rho(y b)}{L_\rho(b)} d\mu_\alpha(y) \le 
\kappa \delta^{1-\alpha-\eta} \,.
\end{equation}
Indeed, recall Remark \ref{bgtslow} on 
existence of $K$ finite, such that 
$L_\rho(yb) \le y^{-\eta} L_\rho(b)$ whenever $b \ge yb \ge K$. 
Hence, with $\int_0^\delta y^{-\eta} d\mu_\alpha(y) = c \delta^{1-\alpha-\eta}$
for some $c=c(\alpha,\eta)$ finite, we only need to consider the 
contribution of $y \le \delta \wedge K/b$ to the 
\abbr{lhs} of \eqref{eq:small-y}. Since $\rho \in \cR_\alpha$ is $(0,1]$-valued, 
the latter is at most 
$(\delta \wedge K/b)/\rho(b)$ which for $b \ge b_0$ is further bounded
by $\delta (1 \wedge K/(\delta b)) b^{\alpha +\eta}$, so the elementary 
inequality $(1 \wedge x) \le x^{\alpha + \eta}$ yields \eqref{eq:small-y}. 

\noindent
(b). For $\alpha>1$ taking as $\mu_\alpha$ the probability measure on 
$[1,\infty)$ of density $(\alpha-1) y^{-\alpha}$, the same change of 
variable as in part (a), transforms \eqref{eq:slow_estimate>1} into 
$$
\lim\limits_{b\rightarrow\infty} \sup_{\delta \in (1,\infty)}\Big|
\frac{\int_1^{\delta} \big(\frac{L_\rho(y b)}{L_\rho(b)} - 1 \big) d\mu_\alpha(y)}
{\mu_\alpha([1,\delta])} \Big|= 0 \,.
$$
As in part (a), for bounded above $\delta$ this trivially 
follows from the uniform convergence $L_\rho(yb)/L_\rho(b) \to 1$, and 
since $\mu_\alpha([\delta,\infty))=\delta^{1-\alpha} \to
0$ as $\delta \to \infty$, it suffices to show that 
\begin{equation}\label{eq:large-y}
\lim_{\delta \uparrow \infty} \limsup_{b \to \infty} 
\int_\delta^\infty \frac{L_\rho(y b)}{L_\rho(b)} d\mu_\alpha(y) = 0 \,.
\end{equation}
To this end, we fix $0<\eta<\alpha-1$ and recall that $L_\rho(yb) \le y^{\eta} L_\rho(b)$
whenever $y b \ge b \ge K$. Since $\int_\delta^\infty y^{\eta} d\mu_\alpha(y) \to 0$
for $\delta \to \infty$, this completes the proof of \eqref{eq:large-y} and of the lemma.
\end{proof}

\begin{proof}[Proof of Lemma \ref{slow}]
(a). In case $\alpha\in [0,1)$ it follows by \eqref{eq:slow_estimate<1} (for $a=0$), 
that
\begin{equation}\label{eq:lim-al<1}
\lim_{b\rightarrow\infty}\frac{I_\rho(b)}{b\rho(b)}=\frac{1}{1-\alpha}
\end{equation}
and consequently $I_\rho(\cdot)$ is regularly varying of order $1-\alpha$. 
Turning to show that the increasing function $I_\rho(\cdot)$ is slowly varying 
when $\alpha=1$, it suffices to show that for any $\lambda>1$,
\begin{align}\label{slow_right_order}
\limsup_{T\rightarrow\infty}\Big\{ \frac{I_\rho(\lambda T)-I_\rho(T)}{I_\rho(T)}
\Big\} \le 0.
\end{align}
To this end, fixing $\delta \in (0,1)$, we have that
\begin{align}\label{slow_lower_bound}
I_\rho(T)\ge \int_{\delta T}^{T}\frac{L_\rho(x)}{x}dx\ge 
\log(1/\delta) \inf_{x\in [\delta T,T]} 
\{ L_{\rho}(x) \} \,,
\end{align}
whereas
\begin{align}\label{slow_upper_bound}
I_\rho(\lambda T)-I_\rho(T)=\int_T^{\lambda T}\frac{L_\rho(x)}{x}dx\le 
\log \lambda \, \sup_{x\in [T,\lambda T]} \{L_{\rho}(x) \} \,.
\end{align}
Dividing \eqref{slow_upper_bound} by \eqref{slow_lower_bound} and taking 
$T\rightarrow\infty$, we arrive at the bound
$$
\limsup_{T\rightarrow\infty}
\Big\{ 
\frac{I_\rho(\lambda T)-I_\rho(T)}{I_\rho(T)} \Big\} 
\le \frac{\log \lambda }{\log (1/\delta)} \,.
$$
Taking now $\delta\rightarrow 0$ yields 
\eqref{slow_right_order} and thereby that $I_\rho(\cdot)$ is slowly varying. 
Finally, since $I_\rho(\infty)<\infty$ when $\alpha>1$, the function $I_\rho(\cdot)$ 
is then (trivially) slowly varying at $\infty$.\\

Proceeding to establish \eqref{eq:lem-2.1_lower_bound}, by the regular variation 
of $\rho(\cdot)$ we have that for any $M > 0$, 
$$
\sup_{\ell \ge 1} \Big\{ 
\frac{M\rho(\ell M)}{\int_{\ell M}^{(\ell+1)M}\rho(t)dt} \Big\} 
\le \sup_{x \ge M} \sup_{\theta \in [1,2]} \Big\{ \frac{\rho(x)}{\rho(\theta x)} \Big\} 
=:\kappa (M) \,, 
$$
with $\kappa(\cdot)$ non-increasing, hence uniformly bounded by universal 
$\kappa_\star$ finite 
(on some $[M_0,\infty)$). Consequently, for any $M \ge M_0$ and $n \ge 1$,
\begin{equation}
\label{eq:univ-bd}
M\sum_{\ell=1}^{n}\rho(\ell M) \le \kappa_\star [I_\rho((n+1)M)-I_\rho(M)] \,, 
\end{equation}
and \eqref{eq:lem-2.1_lower_bound} follows by the regular variation 
of $I_\rho(\cdot)$. If $I_\rho(\infty)<\infty$, the \abbr{rhs} of 
\eqref{eq:univ-bd} goes to zero when $n \to \infty$ followed by $M \to \infty$,
thus yielding \eqref{eq:lem-2.1_lower_bound_new}. 
 
\smallskip
\noindent
(b). Fixing $\epsilon>0$, by the regular variation of $\rho$ 
we have that for any $M \ge M_\epsilon$ and $\ell\ge K_\epsilon$
$$
\Big|\frac{\int_{(\ell-1) M}^{\ell M}\rho(t)dt}{M\rho(\ell M)}-1\Big|\le \sup_{\lambda\in [1-1/\ell,1]}\Big|\frac{\rho(\lambda \ell M)}{\rho(\ell M)}-1\Big|\le \epsilon \,.
$$
Thus, setting $n:=\lceil T/M\rceil$ we have for any $M \ge M_\epsilon$, 
\begin{equation}\label{eq:ubd-mu}
I_\rho(T) \le \sum_{\ell=1}^{n}\int_{(\ell-1) M}^{\ell M} \rho(t)dt \le 
I_\rho(K_\epsilon M) + (1+\epsilon) M \sum\limits_{\ell=1}^n \rho(\ell M) \,.
\end{equation}
If $\alpha\in (0,1]$ then the regularly varying $I_\rho(\cdot)$ has
order $1-\alpha<1$, hence $I_{\rho}(KM)/M \to 0$ when $M \to \infty$
and $K$ is fixed. From \eqref{eq:lim-al<1} the same holds even
for $\alpha=0$, provided $\rho(x) \to 0$. Note that when $\alpha<1$ 
necessarily $I_\rho(\cdot)$ diverges, and our hypothesis extends 
this conclusion to the case of $\alpha=1$. Thus, fixing $\mu>0$ we 
have that $M(T)=\mu I_\rho(T) \to \infty$ when $T \to \infty$. In 
particular, dividing both sides of \eqref{eq:ubd-mu} by 
$M = \mu I_\rho(T)$, then taking $T \to \infty$, yields
$$
\frac{1}{\mu} \le (1+\epsilon) 
\liminf_{T\rightarrow\infty}{\sum\limits_{\ell=1}^n \rho(\ell M)}\,.
$$
Taking now $\epsilon \downarrow 0$ establishes the lower bound of \eqref{eq:lem-2.1}. 
The same reasoning we have used in deriving \eqref{eq:ubd-mu}, leads also to  
\begin{equation}\label{eq:lbd-mu}
I_\rho(n M) \ge I_\rho(K_\epsilon M) + (1-\epsilon) M \sum_{\ell=K_\epsilon+1}^n
\rho(\ell M) \,.
\end{equation}
We divide both sides by $M=\mu I_\rho(T)$, then take $T \to \infty$ followed 
by $\epsilon \downarrow 0$. This in turn results with the corresponding 
upper bound of \eqref{eq:lem-2.1}, since by \eqref{eq:univ-bd}, upon 
fixing $K<\infty$,
$$
\limsup_{M \to \infty} \sum_{\ell=1}^K \rho(\ell M) \le \kappa_\star 
\limsup_{M \to \infty} \frac{I_{\rho}((K+1)M)}{M} = 0 \,,
$$
while by \eqref{eq:lim-al<1}, $n \ge \mu^{-1} T/I_\rho(T) \to \infty$ and hence 
$$
1 \le \frac{I_\rho(nM)}{I_\rho(T)} \le \frac{I_\rho(n M)}{I_\rho((n-1) M)} 
$$ 
where the \abbr{rhs} converges to $1$ when $T \to \infty$,
due to the regular variation of $I_\rho(\cdot)$.
\end{proof}

\begin{proof}[Proof of Lemma \ref{easy}]
By (\ref{nonsum2}) we have the existence of $C<\infty$ and $\eta\in (0,1)$ 
such that for any $\tau_0$ large enough,
\begin{equation}\label{eq:A-univ-bd}
\sup_{\tau \ge \tau_0} \sup_{t \ge \tau/\eta} A(t,t+\tau) \le C \rho(\tau_0)\,.
\end{equation}
Fixing $\epsilon>0$ we take $\tau_0$ large enough to assure that 
$C \rho(\tau_0) \le \epsilon$ (which is always possible since
$\rho(\tau) \rightarrow 0$ when $\tau\rightarrow\infty$). 
For such $\tau_0$ and $n := \lceil M /\tau_0 \rceil$ we set 
$$
t_i:=s+(i-1)\tau_0 \in [s,s+M] \,, \;\; i=1,\ldots,n \,,
$$ 
noting that if $s \ge M/\eta$ then for any $1 \le i,j \le n$,
$$
t_i \ge s \ge \frac{M}{\eta} \ge \frac{(n-1) \tau_0}{\eta} \ge \frac{|t_i-t_j|}{\eta} 
$$ 
and consequently, by \eqref{eq:A-univ-bd}
$$
\E [Z(t_i) Z(t_j)] = A(t_i,t_j) \le C \rho(\tau_0) \le \epsilon \;.
$$
By Slepian's lemma and the union bound, 
we then have for i.i.d. standard normal $\{X_i\}_{i=0}^n$,
any $r \in \R$, $s \ge M/\eta$ and $\epsilon < 5/9$,
\begin{align}
\P(\sup_{t\in[s,s+M]}\{Z(t)\}<r) &\le 
\P(\sup_{i=1}^n \{Z(t_i) \} <r) \le 
\P(\sup_{i=1}^n \{\sqrt{1-\epsilon} X_i + \sqrt{\epsilon} X_0 \} < r)  \nonumber \\
&\le \P\big(X_0<- r \epsilon^{-1/2}\big)
+\P\big( X_1 < 3 r\big)^n \,.
\label{nn2_new}
\end{align} 
Setting $r=\sqrt{\delta \log M}$ we note that for $\delta<0.1$ and all $M$ large enough
$$
\P(X_1 < 3 r)^n \le e^{-n \P(X_1 \ge 3r)} \le e^{-\sqrt{M}} \,.
$$
Thus, 
from \eqref{nn2_new} we deduce that 
$$
\limsup_{M\rightarrow\infty}\frac{1}{\log M} 
\sup_{s \ge M/\eta}
\log\P(\sup_{t\in[s,s+M]}\{Z(t)\}<\sqrt{\delta\log M}) \le 
-\frac{\delta}{2\epsilon} 
$$
and taking $\epsilon \downarrow 0$ results with the desired conclusion \eqref{eq:lem-2.2}. 
\end{proof}

\section{Proof of Theorem \ref{all}}\label{sec-3}

Part (c) of Theorem \ref{all} relies on Theorem \ref{gen2} whereas
parts (a) and (b) follow from \cite[Theorem 1.6]{DM}. For the 
latter task we 
extend the scope of \cite[Lemma 1.8]{DM} 
to non-stationary $A_k(\cdot,\cdot)$ and, for fully handling the $\gamma=2$ 
case, relax the uniform correlation tail decay requirement of \cite[(1.15)]{DM}.
\begin{lem}\label{clm3}
Suppose $\{Z^{(k)}_t\}$, for $1\le k\le \infty$, are centered Gaussian processes 
on $[0,\infty)$, of {non-negative} covariance $A_k$ 
normalized to have $A_k(s,s)=1$ for all $s\ge 0$, such that  
$Z^{(\infty)}_t$ is a stationary process 
and $A_k(s,s+\tau) \to A_\infty(0,\tau)$ when $k \to \infty$, 
uniformly in $s \ge 0$. Suppose
\begin{equation}\label{eq:weak-115}
\limsup_{k,\tau\rightarrow\infty} \sup_{s\ge 0} \Big\{ 
\frac{A_k(s,s+\tau)}{\wt{\rho}(\tau)} \Big\} < \infty
\end{equation}
for some integrable $\wt{\rho} \in \cR_\alpha$, $\alpha \ge 1$, and in addition 
{$A_\infty(0,\tau)$ is non-increasing}, such that
\begin{align}\label{eq:dm-120}
a^2_{h,\theta}:=\inf_{0<t\le h}\left\{\frac{A_\infty(0,\theta t)-A_\infty(0,t)}{1-A_\infty(0,t)}\right\}>0,
\end{align}
and there exists $\eta>1$ such that
\begin{align}\label{eq:dm-123}
\limsup_{u\downarrow0}|\log u|^{\eta}\sup_{1\le k\le \infty,s\ge 0,\tau \in [0,u]}(1-A_k(s,s+\tau))<\infty.
\end{align}
Then we have
\begin{equation}\label{eq:dm-118}
- \lim_{k,T \to \infty} \frac{1}{T}\log \P(\sup_{t\in[0,T]}\{Z^{(k)}_t\}<0)
= b(A_\infty) \,.
\end{equation}
\end{lem}
\begin{proof} The statement \eqref{eq:dm-118} is shown in \cite[Theorem 1.6]{DM} 
to hold under the following assumptions:
\begin{align}
\label{eq:dm-115}
&\limsup\limits_{k,\tau\rightarrow\infty}\sup_{s\ge 0}\left\{\frac{\log A_k(s,s+\tau)}{\log \tau}\right\}<-1,\\
\label{eq:dm-116}-\limsup_{M\rightarrow\infty}\frac{1}{M}\log &\P(\sup_{t\in [0,M]}Z_t^{(\infty)}<M^{-\eta})=b(A_\infty)\text{ for all }\eta>0,
\end{align}
and there exists $\zeta>0,M_1<\infty$ such that for any $z\in [0,\zeta]$ we have
\begin{align}\label{eq:dm-117}
\notag\P(\sup_{t\in [0,M]}Z_t^{(\infty)}<z)\le &\liminf_{k\rightarrow\infty}\inf_{s\ge 0}\P(\inf_{t\in [0,M]}Z_{s+t}^{(k)}<z)\\
\le &\limsup_{k\rightarrow\infty}\sup_{s\ge 0}\P(\inf_{t\in [0,M]}Z_{s+t}^{(k)}<z)\le \P(\sup_{t\in [0,M]}Z_t^{(\infty)}\le z).
\end{align}
 We verify that both
\eqref{eq:dm-116} and \eqref{eq:dm-117} hold here, then adapt
the proof of \cite[Theorem 1.6]{DM} to apply also when 
$\alpha=1$ in \eqref{eq:weak-115} (while 
\eqref{eq:dm-115} follows from \eqref{eq:wh-a-rho} if $\alpha>1$).

It follows from the proof of \cite[Lemma 1.8]{DM}, that \eqref{eq:dm-123} yields the a.s. continuity of 
$s \mapsto Z^{(k)}_s$ for $1 \le k \le \infty$, and that for any $M<\infty$,
the collection $\{Z^{(k)}_{s+\cdot}, \, k \in \N, s \ge 0\}$ is uniformly 
tight in the space $\cC [0,M]$ of continuous functions on $[0,M]$, 
equipped with the topology of uniform convergence. 
This and $A_k(s,s+\cdot) \to A_\infty(0,\cdot)$ uniformly in $s$, 
result with \eqref{eq:dm-117}.
Indeed, the failure of \eqref{eq:dm-117} amounts to 
having $M<\infty$, $z \in \R$, $\epsilon>0$, 
$k_n \uparrow \infty$ and $s_n \ge 0$, such that either    
\begin{align}\label{c2}
& \inf_n \P(\sup_{t\in[0,M]}\{Z^{(k_n)}_{s_n+t}\}<z) \ge 
\P(\sup_{t\in[0,M]}\{Z^{(\infty)}_t\}\le z)+\epsilon \,,
\nonumber \\
\textrm{or} &  \\
& \sup_n \P(\sup_{t\in[0,M]}\{Z^{(k_n)}_{s_n+t}\}<z) \le 
\P(\sup_{t\in[0,M]}\{Z^{(\infty)}_t\} < z) -\epsilon \,.\nonumber
\end{align}
Since $A_{k_n}(s_n,s_n+\cdot) \to A_\infty(0,\cdot)$,
all f.d.d.-s of the Gaussian processes
$\{Z^{(k_n)}_{s_n+\cdot}\}$ 
converge to those of $Z^{(\infty)}_{\cdot}$.
Thus $Z^{(\infty)}_\cdot$ is the limit in distribution on $\cC [0,M]$ 
of $\{Z^{(k_n)}_{s_n+\cdot}\}$ and necessarily
$$
\sup_{t\in[0,M]}\{Z^{(k_n)}_{s_n+t}\}\stackrel{d}{\rightarrow}
\sup_{t\in[0,M]}\{Z^{(\infty)}_t\}\,,
$$
in contradiction with \eqref{c2}. 

Next, recall \cite[Theorem 3.1(iii)]{LS2}, 
that for non-increasing $\tau \mapsto A_\infty(0,\tau)$, \eqref{eq:dm-116}  
yields the continuity of $\varepsilon \mapsto b(A_\infty;\varepsilon)$, where 
\begin{align}\label{c3}
b(A_\infty,\varepsilon):=-\lim_{T\rightarrow\infty}\frac{1}{T}\log \P(\sup_{t\in [0,T]}Z_t^{(\infty)}<\varepsilon)
\end{align}
exists by Slepian's lemma, which in particular verifies the weaker condition \eqref{eq:dm-116} as well.

It thus remains to modify the proof of \cite[Theorem 1.6]{DM} to work under the assumption \eqref{eq:weak-115} instead of \eqref{eq:dm-115}.
Since the lower bound of \cite[Theorem 1.6]{DM} does not involve \cite[(1.15)]{DM} 
it suffices to adapt the proof of the matching upper bound. 
To this end, by \eqref{eq:weak-115} and the regular variation of $\widetilde{\rho}$,
there exist $k_0$, $\tau_0$ and $C$ finite such that $\wt{\rho}$ is 
non-increasing on $[\tau_0,\infty)$ and
\begin{equation}\label{eq:Ak-bd}
A_k(s,t) \le C \wt{\rho}(|s-t|), \qquad 
\forall k\ge k_0, s \ge 0, |t-s|\ge \tau_0 \,.
\end{equation}
Fixing $\delta>0$ and $M \ge \tau_0/\delta$, consider, 
as in \cite[proof of Theorem 1.6]{DM}, 
a maximal collection $\cJ_T$ of $N
$ 
intervals $I_i \subset [0,T]$ of length $M$ each, which are 
$\delta M$-separated. Then, setting 
\begin{equation}\label{eq:gam-def}
\gamma=\gamma(\delta M):=4 C \sum_{i=1}^\infty \wt{\rho}(i \delta M) \,,
\end{equation}
and the symmetric $N$-dimensional matrix ${\bf B}=\{B(i,j)\}$ with   
$B(i,i)=1$ and otherwise $B(i,j)=\frac{C}{\gamma} \wt{\rho}(|i-j| \delta M)$ 
non-increasing 
in  $|i-j|$, we have that 
%
$$
\max_{1\le i\le N} \sum_{j\ne i} B(i,j) \le \frac{2C}{\gamma}\sum_{i=1}^\infty 
\wt{\rho}(i \delta M) \le \frac{1}{2}\,.
$$
Hence,
all eigenvalues of ${\bf B}$ lie within $[1/2,3/2]$. Further, 
if $s \in I_i$ and $t \in I_j$, then $A_k(s,t) \le \gamma B(i,j)$  
by \eqref{eq:Ak-bd} and the monotonicity of $\wt{\rho}$. Consequently, 
the relation of \cite[(2.3)]{DM} holds for any $s,t \in \cJ_T$. The latter
allows us to proceed along the derivation of \cite[(2.4)-(2.6)]{DM}, 
except for replacing the terms $\gamma^\delta$, from the \abbr{rhs} of 
\cite[(2.4)]{DM} onward, by $\varepsilon>0$ (independent of $\gamma$). 
We thus deduce the following variant of \cite[(2.6)]{DM}, 
\begin{align}\label{c_final}
\limsup\limits_{k,T \rightarrow\infty} \frac{1}{T}\log \mathbb{P}( & \sup_{t\in [0,T]}
\{Z^{(k)}_t\}<0) \nonumber \\
\leq &\frac{1}{M(1+\delta)}\log\Big[\mathbb{P}(\sup_{t\in[0,M]}\{Z_t^{(\infty)}\}< 3 \varepsilon)+\sqrt{3}\mathbb{P}(X_1\geq \sqrt{2/3}\varepsilon\gamma^{-1/2})\Big]\,.
\end{align}
Here $X_1$ is standard normal and $M \gamma(\delta M) \to 0$ 
when $M \to \infty$
(by \eqref{eq:lem-2.1_lower_bound_new} and \eqref{eq:gam-def}), hence
\begin{equation}\label{eq:x1-bd}
\limsup\limits_{M\rightarrow\infty}\frac{1}{M}\log  
\mathbb{P}(X_1\geq \sqrt{2/3}\varepsilon\gamma^{-1/2})\leq -
\frac{\varepsilon^2}{6}\liminf_{M\rightarrow\infty} (M\gamma)^{-1} =-\infty \,,
\end{equation}
so in the limit $M \to \infty$ the \abbr{rhs} of \eqref{c_final} 
is at most $-b(A_\infty;3 \varepsilon)/(1+\delta)$, for $b(A_\infty;\cdot)$ of 
\eqref{c3}. 
Thus, considering 
$\varepsilon,\delta\downarrow 0$ yields the upper bound of \cite[Theorem 1.6]{DM}.

Finally, from Lemma \ref{clm1} the events 
$\{\sup_{t\in[0,M]}\{Z_t^{(\infty)}\}< 0\}$
have positive probability, hence $b(A_\infty;0)$ is finite (by 
the non-negativity of $A_\infty$ and Slepian's lemma). Further, in view of
\eqref{eq:x1-bd} the \abbr{rhs} of \eqref{c_final} is strictly negative 
for $M$ large enough, hence its \abbr{lhs}, namely $-b(A_\infty;0)$ is also 
strictly negative.
\end{proof}

\begin{proof}[Proof of part (e) of Corollary \ref{sak_improved}:]
This is a direct application of 
Lemma \ref{clm3}. 
Indeed, 
the relation between occupation times of ${\bf 0}$ 
by  $S_u^{(q)}$, and the number of returns to ${\bf 0}$
by the corresponding embedded discrete time random walk,
implies that   
$$
I_{\rho^{(q)}}(\infty) - I_{\rho^{(q)}}(\tau)= \E[\,G^{(q)}_{N_\tau}\,] \,,
$$ 
for a unit rate Poisson process $\{N_\tau\}$.
Thus, here the auto-correlation of \eqref{eq:Arho-def} is 
\begin{equation}\label{eq:A-q-def}
\overline{C}_{\rho^{(q)}}(0,\tau)=\frac{1}{G^{(q)}_0} \E[\,G^{(q)}_{N_\tau}\,] \,.
\end{equation}
Clearly, $G^{(q)}_0 - 1 = G_1^{(q)} 
\ge G^{(q)}_k\ge G^{(q)}_{k+1}$ for all $k \ge 1$, resulting with 
the bounds 
\begin{equation}\label{eq:bd-Gnq}
1-\P(N_\tau\ge 1) \le \overline{C}_{\rho^{(q)}}(0,\tau)\le 
1-\frac{1}{G^{(q)}_0} \P(N_\tau\ge 1)\,,
\end{equation}
so by the assumed convergence to one of 
$G^{(q_d)}_0$ we have that 
$$
\lim_{d\rightarrow\infty} 
\overline{C}_{\rho^{(q_d)}} (0,\tau) = 1-\P(N_\tau\ge 1)=e^{-\tau}.
$$
Recall that the stationary Ornstein-Uhlenbeck (\abbr{OU}) process
has persistence exponent $b=1$, continuous sample path and the 
correlation function $e^{-|\tau|}$ for which holds.
In addition, from the uniform lower bound on the \abbr{lhs} of \eqref{eq:bd-Gnq}, 
$$
\lim_{u \to 0} |\log u|^2 \sup_d (1-\overline{C}_{\rho^{(q_d)}}(0,u)) = 0 \,,
$$
so \eqref{eq:dm-123} holds as well. 
  Finally, from \eqref{eq:A-q-def}
$$
\overline{C}_{\rho^{(q_d)}}(0,\tau) \le \P(N_\tau \le \tau/2) + G^{(q_d)}_{\tau/2} \,,
$$
hence \eqref{eq:weak-115} follows from the assumed uniform tail bound 
$G^{(q_d)}_{\tau/2} \le \kappa \tau^{-2}$. 
\end{proof}

Equipped with Lemma \ref{clm3} we proceed to establish Theorem \ref{all}.

\begin{proof}[Proof of Theorem \ref{all}]$~$ 

\noindent
(a). To prove \eqref{sum}, we apply Lemma \ref{clm3} for the 
normalized, centered Gaussian processes $\{Z^{(k)}_t\}$, $1 \le k \le \infty$,
of correlation functions
$$
A_k(s,t):=C_\rho(s+k,t+k) \,,\qquad A_\infty(s,t):=\overline{C}_{\rho}(s,t),
\quad  s,t\ge 0\,.
$$
Specifically, from \eqref{eq:Arho-def} we see that $A_\infty(0,\tau)$
is non-increasing and $A_k(s,s+\tau) \to A_\infty(0,\tau)$ as $k \to \infty$, 
uniformly in $s \ge 0$. Further, 
with $\rho \in \cR_\gamma$ uniformly bounded 
and $I_\rho(\cdot)$ strictly positive, non-decreasing, 
it follows from \eqref{A2} that for $s \ge 1$, $\tau>0$, 
$$
1 - C_\rho(s,s+\tau) \le \frac{I_\rho(2s+2\tau)-I_\rho(2s+\tau) + I_\rho(\tau)}{I_\rho(2s)}
\le \frac{2 \sup_{x \ge 0} \{\rho(x)\}}{I_\rho(2)} \tau \,.
$$
Thus $t \mapsto Y_\rho(t)$ is a.s. continuous on $[1,\infty)$ and with the 
preceding holding for $\overline{C}_\rho(\cdot)$, so does \eqref{eq:dm-123}.
Since $C_\rho(\cdot)$ is non-negative, by Slepian's lemma we have that for any 
$k\in (1,T)$,
\begin{align*}
\P(\sup_{t\in[k,T+k]}\{Y_\rho(t)\}<0)\ge 
\P(\sup_{t\in[1,T+k]}\{Y_\rho(t)\}<0)\ge \P(\sup_{t\in[1,k]}\{Y_\rho(t)\}<0)\P(\sup_{t\in[k,T+k]}\{Y_\rho(t)\}<0)
\end{align*}
and with $Y_\rho(t)$ continuous the first term on the \abbr{rhs} is strictly 
positive (see Lemma \ref{clm1}). It thus suffices to confirm that 
\begin{align}\label{chk}
- \lim\limits_{k\rightarrow\infty}\lim\limits_{T\rightarrow\infty}
\frac{1}{T}\log \P(\sup_{t\in[k,T+k]}\{Y_\rho(t)\}<0)=b(\overline{C}_\rho)\,.
\end{align}
With $Z^{(k)}_t=Y_\rho(k+t)$, the identity \eqref{chk} is merely \eqref{eq:dm-118}. We
thus complete the proof upon verifying the remaining two assumptions of Lemma \ref{clm3}, 
first showing that \eqref{eq:weak-115} holds for the integrable 
$\wt{\rho}(s)=s \rho(s)$, then establishing the positivity of 
$a^2_{h,\theta}(\overline{C}_\rho)$ of \eqref{eq:dm-120}.
Turning to the first task, setting $s_k=2s+2k \ge 2$ note that for any $s \ge 0$,
\begin{equation}\label{eq:bd-Crho}
A_k(s,s+\tau)=\frac{I_\rho(s_k+\tau)-I_\rho(\tau)}{\sqrt{I_\rho(s_k+2\tau)I_\rho(s_k)}}\le \frac{I_\rho(\infty)-I_\rho(\tau)}{I_\rho(2)} \,,
\end{equation}
out of which we get \eqref{eq:weak-115}, since by \eqref{eq:slow_estimate>1} 
(for $a \uparrow \infty$ and $\alpha=\gamma>1$), 
\begin{equation}\label{eq:lim-Irho}
\frac{I_\rho(\infty)-I_\rho(\tau)}{\tau \rho(\tau)} \to \frac{1}{\gamma-1} < \infty \,.
\end{equation}
Next, setting $g(\theta,\tau):=\int_\theta^1 \rho(\tau y) dy$ we have from 
\eqref{eq:Arho-def} that 
$$
a^2_{h,\theta} (\overline{C}_\rho)
=\inf_{0<\tau<h} 
\Big\{\frac{g(\theta,\tau)}{g(0,\tau)} \Big\}
\ge (1-\theta) \frac{\inf_{x \in [0,h]}\rho(x)}{\sup_{x\in[0,h]}\rho(x)} > 0 \,,
$$
by our hypothesis that $\rho\in\cR_\gamma$ is uniformly bounded 
away from zero on compacts.

\medskip
\noindent
(b). Considering the Lamperti transformation $t=e^v$ on $[0,V]$ (where 
$T=e^V$), similarly to part (a), due to Lemma \ref{clm1} and the sample path 
continuity of $Y_\rho(\cdot)$ we establish \eqref{sum2} upon showing that 
\begin{align}\label{chk-b}
- \lim\limits_{k\rightarrow\infty}\lim\limits_{V\rightarrow\infty}\frac{1}{V}
\log \P(\sup_{v\in[0,V]}\{Y_\rho(e^{v+k})\}<0)=b(C^\star_\gamma) \in (0,\infty) \,.
\end{align}
The identity \eqref{chk-b} is merely \eqref{eq:dm-118} for the 
centered Gaussian processes $Z^{(k)}_v = Y_\rho(e^{v+k})$ of correlation functions
$$
A_k(v,u):=C_\rho(e^{v+k},e^{u+k}),\quad A_\infty(v,u):=C^\star_\gamma(v,u), 
\quad v,u \ge 0\,.
$$
Thus, \eqref{chk-b} follows once we verify all the assumptions of Lemma \ref{clm3},
at least for all $k \ge k_0$ finite. 
To this effect, for $\gamma \in [0,1)$ we have from \eqref{eq:Agam-def} that 
$A_k(v,v+\tau) \to A_\infty(0,\tau)$, uniformly over $v \ge 0$, with 
$A_\infty(0,\tau)$ non-increasing. Further, for any $\theta \in (0,1)$, 
setting $g(0)=1-\theta^{1-\gamma}>0$ makes 
$$
g(\tau):=
\frac{C^\star_\gamma(0,\theta \tau)-C^\star_\gamma(0,\tau)}{1-C^\star_\gamma(0,\tau)}\,,
$$
a continuous and strictly positive function on $[0,\infty)$. Thus, 
$a^2_{h,\theta}(C^\star_\gamma)$ being the infimum 
of $g(\tau)$ over $[0,h]$ must be positive, and so \eqref{eq:dm-120} holds.
\newline
Turning next to verify \eqref{eq:weak-115}, setting $v_k:=e^{v+k} \ge e^k$ we have 
that 
\begin{equation}\label{eq:Ak-uk}
A_k(v,v+\tau)=\frac{I_\rho((e^\tau+1)v_k)-I_\rho((e^\tau-1)v_k)}{\sqrt{I_\rho(2v_k)
I_\rho(2 e^\tau v_k)}} \,, \qquad \forall v,\tau \ge 0\,. 
\end{equation}
Hence, three applications of \eqref{eq:slow_estimate<1} with $\alpha=\gamma$, 
for $b=(e^{\tau}+1)v_k$, 
$b=2e^{\tau} v_k$ and $b=2v_k$, 
yield 
that 
\begin{align}\label{nn1}
 \lim\limits_{k\rightarrow\infty}\sup_{v,\tau\ge 0}\Big| R_\rho (v_k,\tau)  
 \frac{A_k(v,v+\tau)}{A_\infty(0,\tau)}-1\Big|=0 \,,
\end{align}
where by the eventual monotonicity of 
$x^{\pm 2\eta} L_\rho(x)$ (see Remark \ref{bgtslow}), we further get that
\begin{equation}\label{nn2}
R_\rho(v,\tau) := 
\frac{\sqrt{L_\rho(2v)L_\rho(2 e^\tau v)}}{L_\rho((e^\tau+1) v)} 
\in (e^{-\tau \eta},e^{\tau \eta}) \,,
\end{equation}
for any $\eta>0$, $v \ge v_0(\eta)$ and all $\tau \ge 0$. 
Since $\gamma \mapsto C^{\star}_\gamma(0,\tau)$ is non-increasing, we have that 
for any $\gamma \in [0,1)$,  
$$
A_\infty(0,\tau) \le C^{\star}_0 (0,\tau) = e^{-|\tau|/2} \,.
$$    
Combining this with \eqref{nn1} and \eqref{nn2} (say, for $\eta=1/4$), 
we deduce that $A_k(v,v+\tau) \le 2 e^{-\tau/4}$ for any $k \ge k_0$ 
and all $v,\tau \ge 0$, which is more than enough for \eqref{eq:weak-115}.
\newline
It thus remains to verify \eqref{eq:dm-123}. To this effect, set $\xi:=(1-e^{-\tau})/2 \le \tau$ and 
\begin{equation}\label{eq:F-def}
f(\xi;b,\tau) := \frac{I_{\rho} (b) - I_\rho((1-\xi)b) + I_\rho(\xi b)}
{I_\rho(e^{-\tau} b)} \,.
\end{equation}
Then, considering $b=2 e^\tau v_k$ in \eqref{eq:Ak-uk}, we find 
that for any $v,\tau \ge 0$ and finite $k \ge 1$,  
$$
1-A_k(v,v+\tau) \le \sup_{b \ge 2 e^{k+\tau} } \{f(\xi;b,\tau)\} \,.
$$
As $1-C^\star_\gamma(0,\tau) \le |\tau|^{1-\gamma}$ for $|\tau|$ small enough, 
we get \eqref{eq:dm-123} upon showing that for $\eta>0$ and 
$\kappa,b_0$ finite,
$f(\xi;b,\tau) \le \kappa \xi^{1-\gamma-\eta}$, uniformly over $\tau \in [0,1]$ 
and $b \ge b_0$. To this end, setting  
$$
F_b(a_1,a_2)=\int_{a_1}^{a_2} \frac{L_\rho(by)}{L_\rho(b)} d\mu_\gamma(y) \,,
$$
for $0 \le a_1 \le a_2 \le 1$ and the measure $\mu_\gamma$ on $[0,1]$ of density 
$(1-\gamma) y^{-\gamma}$, recall \eqref{eq:Fb-limit} that
$$
\lim_{b \to \infty} \sup_{\delta \in [0,1)} \Big|
\frac{F_b(\delta,1)}{1-\delta^{1-\gamma}}  -1 \Big| = 0 
$$ 
and \eqref{eq:small-y} that for some $b_0$ finite,
$$
\sup_{b \ge b_0} \sup_{\delta \in (0,1]} \Big\{ \,
\frac{F_b(0,\delta)}{\delta^{1-\gamma-\eta}} \,\Big\} < \infty \;.
$$ 
Further, we find as in the proof of Lemma \ref{clm2}(a), that
$$
f(\xi;b,\tau) = \frac{F_b(1-\xi,1)+F_b(0,\xi)}{F_b(0,e^{-\tau})}\,,
$$ 
where by the preceding, once $b$ is large enough 
$F_b(1-\xi,1) \le 2 \xi$ and $F_b(0,\xi) \le \kappa \xi^{1-\gamma-\eta}$
for all $\xi$, 
while $F_b(0,e^{-\tau}) \ge F_b(0,e^{-1})$ 
are bounded below away from zero. 

\medskip
\noindent
(c). We get \eqref{nonsum} upon applying Theorem \ref{gen2} for 
the centered Gaussian process $Y_\rho(t)$, $t \in [1,\infty)$
(with non-integrable $\wt{\rho}(s)=s \rho(s) \in \cR_{\gamma-1}$). 
Turning to verify the three hypothesis of Theorem \ref{gen2}, 
recall first that while proving part (a) we 
saw that $t \mapsto Y_\rho(t)$ is a.s. continuous and further showed
that $\sup_{s \ge 1} (1-C_\rho(s,s+\tau))$ decay fast enough 
in $\tau \to 0$ to imply that 
$\E[\sup_{t \in [0,1]} \{Y_\rho(s+t)\}]$ is uniformly bounded in $s \ge 1$.
Next, similarly to \eqref{eq:bd-Crho},
$$
C_\rho(t,t+\tau) \le \frac{I_\rho(\infty)-I_\rho(\tau)}{I_\rho(2)}
$$
and \eqref{nonsum2} follows from \eqref{eq:lim-Irho}. Finally, if
$\tau \in [0,\eta t]$ then by \eqref{eq:slow_estimate>1} with $\alpha=\gamma$, 
we deduce that for $\tau \to \infty$, 
$$
\frac{C_\rho(t,t+\tau)}{\wt{\rho}(\tau)} \ge 
\frac{I_\rho(h\tau)-I_\rho(\tau)}{\wt{\rho}(\tau) I_\rho(\infty)} \to
\frac{1-h^{1-\gamma}}{(\gamma-1)I_\rho(\infty)} \,,
$$
which since $h:=1+2/\eta$ diverges with $\eta \downarrow 0$, yields \eqref{nonsum1}
and thereby proves \eqref{nonsum}. 
\end{proof}

\end{document}